\documentclass[11pt]{article}
\usepackage{amsmath,amsthm,amsfonts,amssymb,amscd, amsxtra,mathtools,mathrsfs,commath,url,easyReview}
\usepackage{cite,url}
\usepackage{graphicx}
\usepackage{graphicx}
\usepackage{epsfig}
\usepackage{subfigure}
\usepackage[title]{appendix}

\usepackage{multirow}
\usepackage{multicol}   
\usepackage[margin=2.7 cm,nohead]{geometry}
\usepackage{bbm}
\usepackage{xcolor}
\usepackage{float}
\usepackage{threeparttablex,color,soul,colortbl,hhline,rotating,booktabs,longtable,multirow,makecell}
\definecolor{lightgray}{gray}{0.95}

\usepackage{amssymb,amsfonts,amstext,amsmath}
\usepackage{graphicx,geometry}
\usepackage{listings}
\usepackage{algorithm}%
\usepackage{algorithmicx}%
\usepackage{algpseudocode}%

\usepackage{hyperref}

\newtheorem{theorem}{Theorem}[section]

\newtheorem{lemma}{Lemma}[section]
\newtheorem{proposition}{Proposition}[section]
\theoremstyle{definition}
\newtheorem{definition}{Definition}[section]

\newtheorem{remark}{Remark}[section]

\newtheorem{assumption}{Assumption}[section]
\numberwithin{equation}{section}
\newcommand{\Hf}{\nabla^2 f}

\begin{document}
	
	\title{Global convergence of a modified BFGS-type method based on function information for nonconvex multiobjective optimization problems}
\author{
	Yingxue Yang\thanks{National Center for Applied Mathematics in Chongqing, Chongqing Normal University, Chongqing,  401331, China, E-mails: {\tt  2022110518022@stu.cqnu.edu.cn}}
	\and Xin Deng\thanks{National Center for Applied Mathematics in Chongqing, Chongqing Normal University, Chongqing,  401331, China, E-mails: {\tt  2022110518004@stu.cqnu.edu.cn}} 
	\and Liping Tang \thanks{National Center for Applied Mathematics in Chongqing, Chongqing Normal University, Chongqing,  401331, China, E-mails: {\tt  tanglipings@163.com}}
}

\maketitle

\vspace{.5cm}

\noindent

\noindent
	{\bf Abstract:} 
	 In this paper, based on function information, we propose a modified BFGS-type method for nonconvex multiobjective optimization problems (MFQNMO). In the multiobjective quasi-Newton method (QNMO), each iteration involves separately approximating the Hessian matrix for each component objective function, which results in significant storage and computational burdens. MFQNMO employs a common BFGS-type matrix to approximate the Hessian matrix of all objective functions in each iteration. This matrix is updated using function information from the previous step. This approach strikes a balance between efficacy and computational cost. We confirm the convergence of the method without relying on convexity assumptions, under mild conditions, we establish a local superlinear convergence rate for MFQNMO. Furthermore, we validate its effectiveness through experiments on both nonconvex and convex test problems.
	
	\vspace{8pt}
	\noindent
	{\bf Keywords:}  Nonconvex multiobjective optimization, Quasi-Newton methods, Wolfe line search, Global convergence, Superlinear convergence.
	
	\vspace{8pt}
	\noindent {\bf AMS subject classifications:}

	\section{Introduction} 
In this paper, we consider the following unconstrained multiobjective optimization problem \eqref{MOP}:

\begin{equation}
	\min_{x\in \mathbb{R}^n}f(x)=\left(f_{1}\left(x\right),f_{2}\left(x\right),\ldots,f_{m}\left(x\right)\right)^\top,\tag{MOP}\label{MOP}
\end{equation}

\noindent where $f_i\left(x\right){:}\mathbb{R}^n\to \mathbb{R}$ is continuously differentiable function, for any $i\in[m]$. Multiobjective optimization is widely used in many research areas, which can be found in engineering \cite{engineering 1,engineering 2}, environmental analysis \cite{evironment 1,evironment 2}, machine learning \cite{machine 1,machine 2}, etc.\\
\indent In 2000, Fliege and Svaiter~\cite{steepset} proposed the steepest descent method for multiobjective optimization (SDMO). Building on their pioneering work, researchers have conducted extensive studies on the first-order descent algorithm for solving \eqref{MOP}, (see e.g. \cite{proximal,BB,conditional,conjugate}). The first-order algorithm is often favored for its lower cost than the second-order method, as it avoids the computational burden of determining the Hessian matrix. However, this advantage comes with a trade-off: first-order methods typically exhibit slower convergence rates, especially ill-conditioned problems. \\
\indent Ill-conditioned problems exhibit large curvature variations at certain points, resulting in highly unstable curvature values. The Newton method~\cite{Newton} is effective in such cases as it captures the local curvature of functions and adjusts the direction more efficiently, leading to a faster rate compared to first-order methods. Fliege et al.~\cite{Newton} proved the superlinear convergence rate of the Newton method. However, despite its efficiency, the computational cost per iteration is significantly higher, particularly for large-scale problems. \\
\indent To reduce the computational costs incurred by the direct calculation of the Hessian matrix, the quasi-Newton method is considered. Unlike single objective optimization, for individual objective functions, the quasi-Newton methods \cite{quasi-Newton_1,quasi-Newton_2} employ different matrices to approximate their Hessian matrices, capturing the local curvature. Numerous variants of the quasi-Newton method have been proposed for multi-objective problems (MOPs). In \cite{nonmonotone quasi-Newton}, the global convergence of the quasi-Newton method is established under average nonmonotone line search. Morovati et al. \cite{review-quasi-Newton} extended both the self-scaling BFGS (SS-BFGS) method and the Huang BFGS (H-BGFS) method for multiobjective optimization problems. Prudente and Souza~\cite{wolf-quasi-Newton} proposed a modified BFGS update formula under nonconvex assumptions, ensuring that the algorithm is well-defined. Furthermore, under the assumption of strong convexity, they establish the global convergence and local superlinear convergence rate of the method with Wolfe line search.

In large-scale multiobjective optimization problems, an increase in the number of objective functions results in an exponential growth in the scale of the matrices that need to be stored and computed, particularly in cases with high-dimensional variable spaces. Consequently, it is impractical to calculate or approximate the Hessian matrix for each objective function separately in large-scale scenarios. To address this, Ansary and Panda \cite{modified} utilized a common matrix to approximate all Hessian matrices in each iteration and establish the global convergence under an aggregated Wolf line search. Lapucci and Mansueto \cite{limited BFGS} extended the limited memory Quasi-Newton approach (L-BFGS) for large-scale MOPs, and the global convergence with R-linear convergence rate is obtained under the strong convexity assumption. However, all the improved algorithms mentioned above require convexity conditions to ensure global convergence. Chen et al.~\cite{VBFGS} proposed a variable metric method for unconstrained MOPs. Based on an assumption weaker than the convexity condition, this method establishes strong convergence. Additionally, they mentioned that solving the quadratic constrained subproblem of the second-order method is more difficult than solving the direction-finding dual problem of the first-order method. In \cite{ Global quasi-Newton}, the global convergence can be guaranteed with Wolf line search in the nonconvex case.  \\
\indent For MOPs, there are relatively few studies on quasi-Newton methods in nonconvex cases. This raises the question of whether alternative quasi-Newton method exists that can strike a balance between global convergence and per-iteration cost for nonconvex MOPs. In this paper, inspired by \cite{Global quasi-Newton, modified}, by utilizing function information, we construct a common modified BFGS-type matrix to approximate the Hessian matrices of all objective functions in each iteration. We proposed a modified BFGS-type method based on function information for nonconvex MOPs (MFQNMO).  We verify that the modified BGFS-type matrix is positive definite for each iteration and establish a global convergence for MFQNMO without convexity assumption. Under some mild conditions, we state the superlinear convergence property of MFQNMO. It also shows superior performance  in numerical experiments. Our work has the following advantages:\\
\indent $\bullet $ Rather than approximating the Hessian matrix for each objective function individually, our method utilizes a single common matrix to capture the curvature information. This approach strategically balances the effectiveness of the optimization algorithm with its computational cost.\\
\indent $\bullet $ For nonconvex functions, most versions of quasi-Newton methods encounter difficulties in achieving convergence. Therefore, the MFQNMO method proposed in this paper for nonconvex unconstrained problems is particularly significant and valuable.\\
\indent $\bullet $ our method's update process for the BFGS-type formula does not rely on any specific parameter selections. Instead, it only utilizes information from prior iterations.\\
\indent The structure of this paper is as follows: In Sect. \ref{sec2:peliminaries}, we give some notations and definitions for our later use. Sect. \ref{sec3:steepest} recalls several multiobjective descent methods, including SDMO, QNMO and MQNMO. In Sect. \ref{sec4:MFQNMO}, we introduce the proposed MFQNMO method and establish its convergence analysis. Numerical results are presented in Sect. \ref{sec5:numerical}. At the end of the paper, we draw some conclusions.

\section{Preliminaries} \label{sec2:peliminaries}
In this section, we recall some notations, definitions and results used in the sequel. Let $\mathbb{R}^n$ be the $n$-dimensional Euclidean space, and $\mathbb{R}_+^n$ denote nonnegative orthant of $\mathbb{R}^n$. $N$ represents the set of natural numbers. For any $x:=(x_1,\ldots,x_n)$ and $y:=(y_1,\ldots,y_n)\in\mathbb{R}^n$, we define:
\begin{eqnarray*}
	x\geqslant y&\Leftrightarrow& x-y\in \mathbb{R}_+^n\setminus\{0\};\\
	x\ngeqslant y&\Leftrightarrow& x-y\notin \mathbb{R}_+^n\setminus\{0\};\\
	x>y&\Leftrightarrow &x-y\in \texttt{int}\mathbb{R}_+^n;\\
	x\not>y&\Leftrightarrow &x-y\notin \texttt{int}\mathbb{R}_+^n.
\end{eqnarray*}
The following are some notations that will be used.\\
\indent $\bullet $ $\left[ m \right]=\left\{ 1,2,\cdots ,m \right\}.$\\

\indent $\bullet $ ${{\Delta }_{m}}=\left\{ \lambda :\sum\limits_{i\in \left[ m \right]}{{{\lambda }_{i}}=1},{{\lambda }_{i}}\geqslant 0,\ i\in \left[ m \right] \right\}.$\\

\indent $\bullet $ $\left\| \cdot \right\|$ is the Euclidean norm in ${{\mathbb{R}}^{n}}.$\\

\indent $\bullet $ $~Jf\left( x \right)\in {{\mathbb{R}}^{m\times n}}$, $\nabla {{f}_{i}}\left( x \right)\in {{\mathbb{R}}^{n}}$ ,$\Hf_i(x)$ the Jacobian matrix, the gradient of ${{f}_{i}}$ and the Hessian matrix at $x$, respectively.\\
\indent  $\bullet $ $\mathcal{D}(x,d) \coloneqq \max\limits_{i \in [m]}\nabla f_i(x)^\top d$, $d\in \mathbb{R}^n.$\\
\indent  $\bullet $ For all  $B\in \mathbb{R}^{n\times n}$, $B \succ 0$ means that $B$  is positive definite, $B \succeq 0$ means that $B$ is positive semi-definite.\\
\indent  $\bullet $ For $x,y\in \mathbb{R}^n$, $\langle x,y\rangle_B:=x^\top B y$ and $\|x\|_B:=\sqrt{\langle x,x\rangle_B}$.\\
\indent  $\bullet $ For $x\in \mathbb{R}$, $\lceil x \rceil$ is the least integer greater than or equal to x.\\
\indent  $\bullet $ $A\preceq B \Leftrightarrow B-A \succeq 0,\quad A,B\in \mathbb{R}^{n\times n} $

\begin{definition} \cite{steepset}
	A point ${{x}^{*}}\in {{\mathbb{R}}^{n}}$ is called the Pareto (weakly Pareto ) solution for \ref{MOP}, if there exists no point $x\in {{\mathbb{R}}^{n}}$ such that $f\left( x \right)\leqslant f\left( {{x}^{*}} \right)$ and $f\left( x \right)\ne f\left( {{x}^{*}} \right)~\left( f\left( x \right)<f\left( {{x}^{*}} \right) \right)$.
\end{definition}

\begin{definition} \cite{steepset}
	A point ${{x}^{*}}\in {{\mathbb{R}}^{n}}$ is called a Pareto critical point for \ref{MOP}, if exists $i\in \left[ m \right]$ such that
	$\left\langle \nabla {{f}_{i}}\left( {{x}^{*}} \right),d \right\rangle \geqslant 0,$ for all $d\in {{\mathbb{R}}^{n}}.$
\end{definition}

\begin{definition} \cite{steepset}
	A vector $d\in {{\mathbb{R}}^{n}}$ is called descent direction for $f$ at ${x}^{*}$, if 
	$\left\langle \nabla {{f}_{i}}\left( {{x}^{*}} \right),d \right\rangle <0,$
	for all $i\in \left[ m \right]$.
	
\end{definition}

\begin{lemma}\cite[Theorem 3.1]{Newton}
	The following statements hold:\\
	\indent (i)  if $x^{\ast}$ is local weak Pareto optimal, then $x^{\ast}$ is a Pareto critical point for $f$;\\
	\indent (ii)  if $f$ is convex and $x^{\ast}$ is Pareto critical for $f$, then $x^{\ast}$ is weak Pareto optimal;\\
	\indent (iii)  if $f$ is strictly convex  and $x^{\ast}$ is Pareto critical for $f$, then $x^{\ast}$ is Pareto optimal.
	
\end{lemma}

\section{Descent methods for MOPs.} \label{sec3:steepest}
In this section, we revisit the steepest descent method~\cite{steepset} and the quasi-Newton method~\cite{quasi-Newton_1} for MOPs.\\

\subsection{The steepest descent method for MOPs.}
For a given $x\in {{\mathbb{R}}^{n}}$, consider the following subproblem \eqref{eq1}: 

\begin{equation}
	\underset{d\in {{\mathbb{R}}^{n}}}{\mathop{\min }}\,\underset{i\in [m]}{\mathop{\max }}\,\left\langle \nabla {{f}_{i}}\left( x \right),d \right\rangle +\frac{1}{2}{{\left\| d \right\|}^{2}}.\label{eq1}
\end{equation}

\noindent Since $\left\langle \nabla {{f}_{i}}\left( x \right),d \right\rangle +\frac{1}{2}{{\left\| d \right\|}^{2}}$ is strongly convex for $i\in \left[ m \right]$, then subproblem \eqref{eq1} has a unique minimizer. We denote that $\theta \left( x \right)$ and $d\left( x \right)$ are the optimal value and the optimal solution in problem \eqref{eq1}, respectively. Therefore,
$$
d\left( x \right)=\underset{d\in {{\mathbb{R}}^{n}}}{\mathop{\text{argmin}}}\,\underset{i\in [m]}{\mathop{\max }}\,\left\langle \nabla {{f}_{i}}\left( x \right),d \right\rangle +\frac{1}{2}{{\left\| d \right\|}^{2}},
$$

\noindent and

$$
\theta \left( x \right)=\underset{d\in {{\mathbb{R}}^{n}}}{\mathop{\min }}\,\underset{i\in [m]}{\mathop{\max }}\,\left\langle \nabla {{f}_{i}}\left( x \right),d \right\rangle +\frac{1}{2}{{\left\| d \right\|}^{2}}.
$$

\noindent The problem can be transformed into the following equivalent optimization problem \eqref{SDQP}, which is easier to solve.
\begin{equation} \label{SDQP}
	\begin{array}{cl}
		\underset{(t,d)\in \mathbb{R}\times \mathbb{R}^{n}}{\text{min}}    & \;  t + \frac{1}{2}\|d\|^2          \\
		\text{s.t.} & \langle\nabla f_{i}(x), d\rangle \leqslant t, \quad i\in [m].
	\end{array}
\end{equation}
\noindent We consider that the dual problem of problem \eqref{SDQP}:

$$\underset{\lambda \in {{\Delta }_{m}}}{\mathop{\max }}\,\underset{\left( t,d \right)\in \mathbb{R}\times {{\mathbb{R}}^{n}}}{\mathop{\min }}\,L\left( \lambda ,\left( t,d \right) \right),$$
\noindent where $L\left( \lambda ,\left( t,d \right) \right)$ is a Lagrange function,
$$L\left( \lambda ,\left( t,d \right) \right)=t+\frac{1}{2}{{\left\| d \right\|}^{2}}+\underset{i\in \left[ m \right]}{\mathop \sum }\,{{\lambda }_{i}}\left( \left\langle \nabla {{f}_{i}}\left( x \right),d \right\rangle -t \right).$$

\noindent According to the KKT condition, we can obtain the steepest descent direction:

$$
d_{SD}\left( x \right)=-\underset{i\in \left[ m \right]}{\mathop \sum }\,{{\lambda }_{i}^{SD}}\left( x \right)\nabla {{f}_{i}}\left( x \right),\label{eq4}
$$

\noindent where $\lambda_{SD} \left( x \right)={{\left( {{\lambda }_{1}^{SD}}\left( x \right),{{\lambda }_{2}^{SD}}\left( x \right),\cdots ,{{\lambda }_{m}^{SD}}\left( x \right) \right)}^\top}$is the solution of the following dual problem

\begin{equation}
	-\underset{\lambda \in {{\Delta }_{m}}}{\mathop{\min }}\,\frac{1}{2}{{\left\| \underset{i\in \left[ m \right]}{\mathop \sum }\,{{\lambda }_{i}}\nabla {{f}_{i}}\left( x \right) \right\|}^{2}},\label{SDDP}
\end{equation}

\noindent owing to strong duality, we have:

$$
\theta_{SD} \left( x \right)=-\frac{1}{2}{{\left\| \underset{i\in \left[ m \right]}{\mathop \sum }\,{{\lambda }_{i}^{SD}}\left( x \right)\nabla {{f}_{i}}\left( x \right) \right\|}^{2}}=-\frac{1}{2}{{\left\| d_{SD}\left( x \right) \right\|}^{2}}\leqslant 0,
$$

$$
\left\langle \nabla {{f}_{i}}\left( x \right),d_{SD}\left( x \right) \right\rangle \leqslant -{{\left\| d_{SD}\left( x \right) \right\|}^{2}}\leqslant 0.
$$
The following lemma describes the proposition of $\theta_{SD} \left( x \right)$ and $d_{SD}\left( x \right).$

\begin{lemma}\label{lem:SDdirection} \cite[Proposition 3.1]{steepset} 
	For problem \eqref{eq1}, $\theta_{SD} \left( x \right)$ and $d_{SD}\left( x \right)$ are the optimal value and solution, respectively. Then,\\
	$(a)$ It can be conclude that the equivalent relation is as follows.\\
	\indent $(i)$ The point $x\in {{\mathbb{R}}^{n}}$ is not Pareto critical;\\
	\indent $(ii)$ $d_{SD}\left( x \right)\ne 0$;\\
	\indent $(iii)$ $\theta_{SD} \left( x \right)<0$.\\
	$(b)$ The function $d_{SD}\left( x \right)$ and $\theta_{SD} \left( x \right)$ are continuous, for all $x\in {{\mathbb{R}}^{n}}.$
\end{lemma}

\subsection{The quasi-Newton methods for MOPs (QNMO).}
Quasi-Newton direction $d_q(x)$ in \cite{quasi-Newton_1} is the optimal solution to the following subproblem:
\begin{equation}\label{quasi-newton}
	\min_{d\in \mathbb{R}^n} \max_{i\in [m]} \nabla f_i (x)^\top d + \frac{1}{2}d^\top B_i(x) d,
\end{equation}
where $B_i(x)\in \mathbb{R}^{n\times n}$, approximates the Hessian matrix $\Hf_i(x).$ For $i\in [m]$, $B_i(x)\succ 0$ such that $\nabla f_i (x)^\top d + \frac{1}{2}d^\top B_i(x) d$ is strongly convex. Then, the subproblem \ref{quasi-newton} has a unique minimizer. We denote that $\theta_q \left( x \right)$ and $d_q\left( x \right)$ be the optimal value and solution in problem \eqref{quasi-newton}, respectively,
$$
d_q(x) = \mathop{\text{argmin}}_{d \in \mathbb{R}^n}\max_{i\in [m] }\nabla f_i (x)^\top d + \frac{1}{2}d^\top B_i(x) d,
$$ 
and
$$
\theta_q(x) = \mathop{\text{min}}_{d \in \mathbb{R}^n}\max_{i \in [m] } \nabla f_i (x)^\top d + \frac{1}{2}d^\top B_i(x) d.
$$
\eqref{quasi-newton} is equivalent to the convex quadratic optimization problem:
\begin{align*} 
	\begin{array}{cl}
		\underset{(t,d)\in \mathbb{R}\times \mathbb{R}^{n}}{\text{min}}    & \;  t          \nonumber \\
		\text{s.t.} & \nabla f_i (x)^\top d + \frac{1}{2}d^\top B_{i}(x) d  \leqslant t,  \quad \forall i\in [m].\nonumber
	\end{array}
\end{align*}
According to the KKT condition, we can obtain the quasi-Newton direction:
\begin{equation}\label{quasi-d}
	d_q(x) = - \left[ \sum_{i=1}^{m} \lambda_i^{q}(x)B_i(x) \right]^{-1} \left(\sum_{i=1}^{m}\lambda_i^{q}(x)\nabla f_i (x)\right),
\end{equation}
where $\lambda_{q} \left( x \right)={{\left( {{\lambda }_{1}^{q}}\left( x \right),{{\lambda }_{2}^{q}}\left( x \right),\cdots ,{{\lambda }_{m}^{q}}\left( x \right) \right)}^\top}$is the solution of the following dual problem

$$
-\underset{\lambda \in {{\Delta }_{m}}}{\mathop{\min }}\,\frac{1}{2}{{\left\| \underset{i\in \left[ m \right]}{\mathop \sum }\,{{\lambda }_{i}}\nabla {{f}_{i}}\left( x \right) \right\|^2_{\left[\underset{i\in \left[ m \right]}{\mathop \sum }\,{{\lambda }_{i}{B_{i}(x)} } \right]^{-1}}}},
$$ 
owing to strong duality, we have
$$
\theta_q(x) = -\frac{1}{2} d_q(x)^{\top} \left[ \sum_{i=1}^{m} \lambda_i^{q}(x)B_i(x) \right] d_q(x).    
$$
\subsection{The modified quasi-Newton methods for MOPs (MQNMO).}
Even if quasi-Newton avoids calculating the second-order Hessian matrix, the per-iteration cost is still expensive. Recently, Ansary and Panda \cite{modified} proposed a modified quasi-Newton method (MQNMO) that utilized a single positive matrix to approximate all the Hessian matrices. The Modified quasi-Newton direction $d(x)$ is the optimal solution to the following subproblem:
\begin{equation}\label{M-newton}
	\min_{d\in \mathbb{R}^n} \max_{i\in [m]} \nabla f_i (x)^\top d + \frac{1}{2}d^\top B(x) d,
\end{equation}
where $B(x)\in \mathbb{R}^{n\times n}$ and  $B(x)\succ 0$. Problem \eqref{M-newton} is equivalent to the convex quadratic optimization problem as follows 
\begin{align*} 
	\begin{array}{cl}
		\underset{(t,d)\in \mathbb{R}\times \mathbb{R}^{n}}{\text{min}}    & \;  t          \\
		\text{s.t.} & \nabla f_i (x)^\top d + \frac{1}{2}d^\top B(x) d  \leqslant t,  \quad \forall i\in [m].
	\end{array}
\end{align*}
Similar to the analysis in the previous section, we can obtain the modified quasi-Newton direction using the KKT condition
\begin{equation}\label{direction}
	d(x) = - B(x)^{-1}\left( \sum_{i=1}^{m}\lambda_i(x)\nabla f_i (x) \right),
\end{equation}
and
\begin{equation}\label{theta}
	\theta(x) = -\frac{1}{2} d(x)^\top B(x)d(x),    
\end{equation}
where $\lambda \left( x \right)={{\left( {{\lambda }_{1}}\left( x \right),{{\lambda }_{2}}\left( x \right),\cdots ,{{\lambda }_{m}}\left( x \right) \right)}^\top}$is the solution of the following dual problem \eqref{MDP}

\begin{equation}
	-\underset{\lambda \in {{\Delta }_{m}}}{\mathop{\min }}\,\frac{1}{2}{{\left\| \underset{i\in \left[ m \right]}{\mathop \sum }\,{{\lambda }_{i}}\nabla {{f}_{i}}\left( x \right) \right\|}_{B(x)^{-1}}^{2}}.\label{MDP}
\end{equation}
According to the above analysis, we can get
\begin{equation}\label{dotz-max}
	\mathcal{D}(x_k,d_k) \leqslant -\|d_k\|_{B_k}^2
\end{equation}
The following rank one DFP formula of generating $B_{k+1}$ is defined in \cite{modified} by
\begin{equation}\label{DFP}
	B_{k+1} = (I-\frac{y_ks_k^\top}{s_k^\top y_k}) B_{k} (I-\frac{s_k y_k^\top}{s_k^\top y_k}) + \frac{y_ky_k^\top}{s_k^\top y_k},
\end{equation}
where $s_k = x_{k+1}-x_{k}$ and $y_k = \sum\limits_{i=1}^{m} \lambda_i^k(\nabla f_i(x_{k+1})-\nabla f_i(x_k)).$

\section{MFQNMO:Modified BFGS-type method based on function information for MOPs.}   \label{sec4:MFQNMO}

In this section, we propose a modified BFGS-type method based on function information for MOPs and analyze its properties. The formal description of the algorithm is as follows. 
\begin{algorithm}[H]
	\caption{MFQNMO}\label{algo1}
	\begin{algorithmic}[1]
		\Require Choose starting point $x^0\in\mathbb{R}^n$. Let $0<\sigma_{1}<\sigma_{2} <1$, $k=0$, $\varepsilon \geqslant 0$ and $B_0\succ 0.$
		\item[Step 1.]  Compute the solution ${{\lambda }_{k}}={{\left(\lambda_1^{k},\lambda_2^{k},\cdots ,\lambda_m^{k} \right)}^\top}$ of problem \eqref{MDP} and ${{d}_{k}}$, $\theta_{k}$ as \eqref{direction}, \eqref{theta}.
		\item[Step 2.] If $\left| {{\theta }_{k}} \right|<\varepsilon $, return $x_{k}.$
		\item[Step 3.] Compute step size $\alpha_{k}>0$ satisfies
		\begin{align} 
			f_i(x_k+\alpha_k d_k) & \leqslant f_i(x_k)+\sigma_1\alpha_k \mathcal{D}(x_k,d_k),  \quad \forall i \in [m], \label{wolf1}  \\ 
			\mathcal{D}(x^k+\alpha_kd_k,d_k) & \geqslant \sigma_2 \mathcal{D}(x_k,d_k), \label{wolf2}
		\end{align}
		\item[Step 4.] Set $x_{k+1} = x_k + \alpha_k d_k$.
		\item[Step 5.] Compute
		\begin{equation}\label{etak}
			\eta_k = \dfrac{y_k^\top s_k}{\|s_k\|^2}, 
		\end{equation}
		\begin{equation}\label{mk}
			m_k = \max \{-\eta_k, 0\}  + \sum_{i=1}^m \lambda_i^k (f_i(x_k)-f_i(x_{k+1})).
		\end{equation}
		\begin{equation}\label{gammak}
			\gamma_k  = y_k + m_ks_k.
		\end{equation}
		where $s_k = x_{k+1}-x_{k}$ , $\mu_i^k = \nabla f_i(x_{k+1})-\nabla f_i(x_k)$ and $y_k = \sum\limits_{i=1}^m \lambda_i^k \mu_{i}^k$. 
		\item[Step 6.]  Update the BFGS-type matrices,
		\begin{equation}\label{MFBFGS}
			B_{k+1} = B_{k} - \dfrac{B_{k}s_ks_k^\top B_{k}}{s_k^\top B_ks_k} +\dfrac{\gamma_k\gamma_k^\top }{\gamma_{k}^\top s_k},
		\end{equation}
		\begin{equation}\label{MFBFGS_inverse}
			B_{k+1}^{-1} = (I-\frac{s_k\gamma_k^\top}{s_k^\top\gamma_k}) B_{k}^{-1} (I-\frac{\gamma_k s_k^\top}{s_k^\top\gamma_k}) + \frac{s_ks_k^\top}{s_k^\top\gamma_k}. 
		\end{equation}
		\item[Step 7.] Set $k\leftarrow k+1$ and go to Step 1.		
	\end{algorithmic} 
\end{algorithm}

\begin{remark}
	In step 1, solving the descent direction through the dual problem of subproblem \ref{MDP}, it is noted that the initial BFGS matrix is selected as a positive definite matrix, usually chosen as the unit matrix. Step 2 is the termination criterion of the algorithm. In step 3, we search the step size through the Wolfe line search. In steps 5-6, we notice that the BFGS-type update formula  \eqref{MFBFGS} and \eqref{MFBFGS_inverse} is different from \cite{Global quasi-Newton}. We use the previous iterative function information to update the $m_k$ adaptively. It can note that 
	\begin{equation}\label{section equation}
		B_{k+1}s_k = y_k + m_ks_k.
	\end{equation}
	
\end{remark}
The following result shows that Algorithm \ref{algo1} is well-defined without convexity assumption. 

\begin{theorem} $x_k$ is not the Pareto critical point, then $B_k$ updated by Algorithm \ref{algo1} is positive definite, for all $k\ge 0$.
\end{theorem}
\begin{proof}
	By the definition of $\gamma_k$, $\eta_k$ and $m_k$, we have
	\begin{align}
		\gamma_k^\top s_k 
		& = y_k^\top s_k + m_k\|s_k\|^2 \nonumber \\ 
		& =\left(\frac{y_k^\top s_k}{\|s_k\|^2} + m_k \right)\|s_k\|^2 \nonumber \\ 
		& = (\eta_k + m_k)\|s_k\|^2 \nonumber \\ 
		& = \left[\eta_k + \max \{-\eta_k, 0\}  + \sum_{i=1}^m \lambda_i^k \left(f_i(x_k)-f_i(x_{k+1})\right)\right]\|s_k\|^2 \nonumber \\ 
		& \geqslant \sum_{i=1}^m \lambda_i^k \left(f_i(x_k)-f_i(x_{k+1})\right)\|s_k\|^2 \label{ieq:zhuhecha}\\ 
		& > 0. \label{eq:gammas_k}
	\end{align}
	Then for all vector $x\in \mathbb{R}^n$,
	\begin{equation}
		x^\top B_{k+1} x = \|x\|_{B_{k}}^2-\dfrac{\langle x,s_k \rangle_{B_{k}}^2}{\|s_k\|_{B_{k}}^2}+\dfrac{(x^\top\gamma_k)^2}{\gamma_{k}^\top s_k},
	\end{equation}
	it follows by the Cauchy-Schwarz inequality and \eqref{eq:gammas_k} that
	\begin{equation}\label{eq:x*B_{k+1} x}
		x^\top B_{k+1} x \geqslant \dfrac{(x^\top\gamma_k)^2}{\gamma_{k}^\top s_k} \geqslant 0. 
	\end{equation}
	
	\noindent Since $x^\top B_{k+1} x = 0$ if and only if $x^\top\gamma_k = 0$ and $x= cs^k$, where $c$ is a constant. It follows that  $cs_k^\top \gamma_{k}=0$. By \eqref{eq:gammas_k}, we can obtain that $c=0$ and $x=0.$ Then for all nonzero vector $x\in \mathbb{R}^n,$
	$$x^\top B_{k+1} x > 0,$$
	we conclude the proof.
\end{proof}
\begin{assumption}\label{assum1}
	For any $z\in \mathbb{R}^n$, the level set $\left\{ x:f(x)\leqslant f(z) \right\}$ is a bounded set. 
\end{assumption}
Next, we show that the Wolf line search in Algorithm \ref{algo1} can terminate at a finite step.
\begin{proposition}\label{pro:step_bound}
	Suppose Assumption \ref{assum1} holds and $d_k$ is the descent direction. Then there exists an interval of values $[\alpha_1, \alpha_2]$, with $0<\alpha_1< \alpha_2$, such that for all $\alpha \in [\alpha_1, \alpha_2]$, inequalities \eqref{wolf1} and \eqref{wolf2} hold.
\end{proposition}
\begin{proof}
	The proof is similar to that in \cite[Proposition 2]{limited BFGS}, we omit it here.
\end{proof}

\begin{lemma}\label{lem:direction}
	Let $\theta(x)$ and $d(x)$ be the optimal value and solution in problem \eqref{M-newton}. Then we have:\\
	\noindent $(a)$ Vector $d(x)\ne 0$ is a descent direction.\\
	\noindent $(b)$ The following conditions are equivalent:\\
	\indent $(i)$ The point $x\in {{\mathbb{R}}^{n}}$ is not Pareto critical;\\
	\indent $(ii)$ $d(x)\ne 0$;\\
	\indent $(iii)$ $\theta(x)<0$.
\end{lemma}
\begin{proof}
	These assertions can be derived by using the same proof demonstrated in the  \cite[Lemma 2.2]{Global quasi-Newton}. 
\end{proof}

\subsection{Global convergence} \label{Glbal}
In this section, we present the global convergence analysis of Algorithm \ref{algo1} without assuming convexity. First, we require the following assumptions.

\begin{assumption}\label{assum2}
	\textit{(i)} For $i \in [m]$, $\nabla f_i(x)$ is Lipschitz continuous with constant $L_i.$ \\
	(ii)For any $x\in \mathbb{R}^{n}$, there exists constant $a$ and $b$ such that
	$$aI_n \preceq B(x) \preceq bI_n.$$
\end{assumption}	
In the subsection, the following properties imply that under the Wolfe line search, the objective function will be significantly decreased along the direction generated by Algorithm \ref{algo1}.

\begin{lemma}\label{lem:cos}
	Let $B_k$ be generated by the BFGS update formula \eqref{MFBFGS} in Algorithm \ref{algo1}. If there exists constants $c_1$ and $c_2$ such that $\frac{\gamma_k^{\top}s_k}{\|s_k\|^{2}} \geqslant c_1 >0$, $\frac{\|\gamma_k\|^{2}}{\gamma_k^{\top}s_k} \leqslant c_2$. Then, for any $p\in (0,1)$, there exists a constant $\delta>0$ such that,  for any $k\in \mathbb{N}$,
	\begin{equation}
		\cos\beta_j = \frac{s_j^{\top}B_js_j}{\|s_j\|\|B_js_j\|} \geqslant \delta,
	\end{equation}
	holds for at least $\lceil p(k+1) \rceil$ values of  $j \in [0,k].$
\end{lemma}	
\begin{proof}
	The proof is similar to that in \cite[Theorem 2.1]{tool}, we omit it here.
\end{proof}

\begin{proposition}\label{pro:dot-bizhi}
	Let $\left\{ {{x}_{k}} \right\}$ be the sequence generated by Algorithm \ref{algo1}. Suppose that Assumption \ref{assum2}  holds. Then, there exists constants $c_1$ and $c_2$ such that
	\begin{equation}\label{bizhi 1}
		\frac{\gamma_k^{\top}s_k}{\|s_k\|^{2}} \geqslant c_1\|d_k\|_{B_k}^{2} ,
	\end{equation}
	\begin{equation}\label{bizhi 2}
		\frac{\|\gamma_k\|^{2}}{\gamma_k^{\top}s_k} \leqslant \frac{c_2}{\|d_k\|_{B_k}^2}.
	\end{equation}
\end{proposition}

\begin{proof}
	From Wolfe line search, \eqref{dotz-max} and  \eqref{ieq:zhuhecha}, we can get
	$$ f_i(x_k)-f_i(x_{k+1}) \geqslant -\sigma_1\alpha_k \mathcal{D}(x_k,d_k) \geqslant \sigma_1\alpha_k\|d_k\|_{B_k}^2 ,$$
	$$\gamma_k^\top s_k \geqslant \sum_{i=1}^m \lambda_i^k \left(f_i(x_k)-f_i(x_{k+1})\right)\|s_k\|^2 .$$
	Let $c_1=\sigma_1\alpha_1>0$, it follows by Proposition \ref{pro:step_bound} that 
	$$\frac{\gamma_k^{\top}s_k}{\|s_k\|^{2}} \geqslant c_1\|d_k\|_{B_k}^{2}.$$
	Next, we consider that inequality \eqref{bizhi 2}. By the definition of the \ref{gammak}, 
	\begin{align}\label{pro:dot-bizhi:1}
		\|\gamma_k\| 
		& \leqslant \|y_k\|+m_k\|s_k\| \nonumber \\
		& \leqslant \left(\frac{\|y_k\|}{\|s_k\|} + m_k\right)\|s_k\| \nonumber \\
		& \leqslant \left[\frac{\|y_k\|}{\|s_k\|} +\max \{-\eta_k, 0\}  + \sum_{i=1}^m \lambda_i^k \left(f_i(x_k)-f_i(x_{k+1})\right) \right]\|s_k\|. 
	\end{align}
	Since $\nabla f_i(x)$ is Lipschitz continuous with constant $L_i$, for $i\in [m]$, we can get
	$$    \frac{\|y_k\|}{\|s_k\|} \leqslant \sum_{i\in [m]}\lambda_{i}^{k}\frac{\|\nabla f_i(x_{k+1})-\nabla f_i(x_k)\|}{\|x_{k+1}-x_k\|} \leqslant \max_{i\in [m]} L_i.$$
	By the Cauchy-Schwarz inequality, 
	$$\frac{|\langle y_k,s_k \rangle|}{\|s_k\|^2} \leqslant \frac{\|y_k\|}{\|s_k\|}.$$
	Therefore, 
	\begin{equation}\label{pro:dot-bizhi:2}
		\frac{\|y_k\|}{\|s_k\|} + \max \{-\eta_k, 0\} \leqslant \frac{2\|y_k\|}{\|s_k\|} \leqslant 2\max_{i\in [m]} L_i.
	\end{equation}
	Assumption \ref{assum1} and Assumption \ref{assum2} (i) imply the existence of  $f_{1} \in \mathbb{R}$ such that $f_{i}(x_k) \geqslant f_1$ for all $k\geqslant 0$ and $i \in [m]$. It follows by \eqref{pro:dot-bizhi:1} and \eqref{pro:dot-bizhi:2} that
	$$\|\gamma_k\| \leqslant (2\max_{i\in [m]} L_i + \max_{i\in [m]}(f_i(x_0)-f_{1}))\|s_k\|,$$
	hence, 
	$$\frac{\|\gamma_k\|^{2}}{\gamma_k^{\top}s_k} \leqslant \frac{c_2}{\|d_k\|_{B_k}^2},$$
	where $c_2 = \frac{(2\max\limits_{i\in [m]} L_i + \max\limits_{i\in [m]}(f_i(x_0)-f_{1}))^2}{c_1}.$
	
\end{proof}
\begin{remark}
	$\beta_k$ is the angle between $s_k$ and $B_ks_k$. By Lemma \ref{lem:cos} and Proposition \ref{pro:dot-bizhi}, we can obtain that the angles $\beta_k$ remain away from 0, for an arbitrary fraction $p$ of the iterates.
\end{remark}

\begin{lemma}\label{lem:D(x_k,d_k)}
	Let $\left\{ {{x}_{k}} \right\}$ be the sequence generated by Algorithm \ref{algo1}. Then,
	\begin{equation}
		\mathcal{D}(x_k,d_k) \leqslant -\cos\beta_k \|d_k\|\|d_{SD}^{k}\|.
	\end{equation}
\end{lemma}
\begin{proof}
	Since $s_k = x_{k+1}-x_k = \alpha_kd_k$, then
	$$\cos\beta_k = \frac{s_k^{\top}B_ks_k}{\|s_k\|\|B_ks_k\|} = \frac{d_k^{\top}B_kd_k}{\|d_k\|\|B_kd_k\|}.$$
	It follows that
	\begin{align*}
		\mathcal{D}(x_k,d_k) 
		& \leqslant -\|d_k\|_{B_k}^2  \\
		& = -\cos\beta_k \|d_k\|\|B_kd_k\| \\
		& = -\cos\beta_k \|d_k\| \left\|\sum_{i \in [m]}\lambda_{i}^k \nabla f_i(x_k)\right\| \\
		& \leqslant -\cos\beta_k \|d_k\| \|d_{SD}^{k}\|.
	\end{align*}
	This completes the proof. 
	
\end{proof}
The following result establishes that the sequence $\{x_k,d_k\}$ satisfies the Zoutendijk condition, which is the basis for establishing the global convergence of Algorithm \ref{algo1}. 

\begin{lemma}\label{Zoutendijk condition} 
	Let $\left\{ {{x}_{k}} \right\}$ be the sequence generated by Algorithm \ref{algo1}. Suppose that Assumption \ref{assum1} and Assumption \ref{assum2} hold. Then 
	\begin{equation}
		\sum_{k\geqslant 0}\frac{\mathcal{D}(x_k,d_k)^2}{\|d_k\|^2} < \infty.
	\end{equation}
\end{lemma}
\begin{proof}
	The proof is similar to that in \cite[Proposition 3.3]{conjugate}, we omit it here.
\end{proof}
The following theorem establishes the global convergence of Algorithm \ref{algo1}. 

\begin{theorem}\label{Global}
	Let $\left\{ {{x}_{k}} \right\}$ be the sequence generated by Algorithm \ref{algo1}. Suppose that Assumption \ref{assum1} and Assumption \ref{assum2} hold. Then, the following statements hold.   \\
	\indent \textit{(i)}\quad $\liminf\limits_{k\to\infty}\|d_{SD}^k\| = 0$,\\
	\indent \textit{(ii)}\quad Every accumulation point of the sequence $\left\{ {{x}_{k}} \right\}$ is Pareto critical.
\end{theorem}

\begin{proof}
	Assume that there exists a constant $\varepsilon >0$ such that
	\begin{equation}\label{contradiction}
		\|d_{SD}^k\| \geqslant \varepsilon, \quad \forall k \geqslant 0.
	\end{equation}
	By the definition of $d_k$ and Assumption \ref{assum2} (ii), we get
	\begin{align*}
		-\frac{1}{2}\|d_k\|^2 
		& = \min\limits_{d\in \mathbb{R}^n}\max\limits_{i\in [m]}{\nabla f_i(x_k)^{\top}d + \frac{1}{2}}\|d\|_{B_k}^2  \\
		& \leqslant \min\limits_{d\in \mathbb{R}^n}\max\limits_{i\in [m]}{\nabla f_i(x_k)^{\top}d + \frac{b}{2}}\|d\|^2 \\
		& \leqslant \frac{1}{b}\min\limits_{d\in \mathbb{R}^n}\max\limits_{i\in [m]}{\nabla f_i(x_k)^{\top}(bd) + \frac{1}{2}}\|bd\|^2 \\
		& = -\frac{1}{2b}\|d_{SD}^k\|^2.
	\end{align*}
	Set $C_1 = \frac{c_1\varepsilon^2}{b}$ and $C_2 = \frac{c_2 b}{\varepsilon^2}$, it follows by Proposition \ref{pro:dot-bizhi} that
	$$ \frac{\gamma_k^{\top}s_k}{\|s_k\|^{2}} \geqslant C_1 >0,$$
	$$\frac{\|\gamma_k\|^{2}}{\gamma_k^{\top}s_k} \leqslant C_2.$$
	Then, by Lemma \ref{lem:D(x_k,d_k)}  and Lemma \ref{lem:cos}, there exist a constant $\delta>0$ and a infinite subsequence of indices $\mathcal{K} \subset \mathbb{N}$ such that 
	$$-\frac{\mathcal{D}(x_k,d_k)}{\|d_k\|} \geqslant \delta \|d_{SD}^k\| \geqslant \delta\varepsilon,\quad \forall k\in \mathcal{K},$$
	it follows that
	$$\sum\limits_{k\in \mathbb{N}}\frac{\mathcal{D}(x_k,d_k)^2}{\|d_k\|^2} \geqslant \sum\limits_{k\in \mathcal{K}}\frac{\mathcal{D}(x_k,d_k)^2}{\|d_k\|^2} \geqslant \sum\limits_{k\in \mathcal{K}} \delta^2\varepsilon^2 = \infty,$$
	this contradicts the result of Lemma \ref{Zoutendijk condition}. Then,
	$$\liminf\limits_{k\to\infty}\|d_{SD}^k\| = 0.$$
	Assumption \ref{assum1} implies there exists a infinite subsequence of indices $K_1 \subset{\mathbb{N}}$ such that $x_k \stackrel{K_1}{\longrightarrow} x^*$ and $ \|d_{SD}^k\| \stackrel{K_1}{\longrightarrow} 0.$ By the Lemma \ref{lem:SDdirection}, $d_{SD}(x)$ is continuous, we can get
	$$d_{SD}(x^*) = 0.$$
	Furthermore, by Lemma \ref{lem:direction}, $x^*$ is Pareto critical.
\end{proof}
We have obtained the weak convergence of the sequence $\{x^k\}$ generated by Algorithm \ref{algo1}. Under strict convexity, the following theorem demonstrates that the sequence $\{x^k\}$ generated by Algorithm \ref{algo1} strongly converges to the Pareto optimal point $x^*.$

\begin{theorem}\label{them:convex_global}
	Let $\{x_k\}$ be the sequence generated by Algorithm \ref{algo1}. Suppose that $f(x)$ is strictly convex, Assumption \ref{assum1} and Assumption \ref{assum2} hold. Then, $\{x_k\}$ converges to a Pareto optimal point $x^*$ of \eqref{MOP}.
\end{theorem}
\begin{proof}
	The proof is similar to that in \cite[Theorem 3.6]{Global quasi-Newton}, we omit it here. 
\end{proof}

\subsection{Superlinear convergence analysis } \label{sec:superlinear}

In this section, we consider the local superlinear convergence of Algorithm \ref{algo1} under Assumption \ref{assum3}. 
\begin{assumption}\label{assum3}
	\indent \textit{(i)} $F$ is twice continuously differentiable.\\
	\indent \textit{(ii)} There exists constants $U,L>0$  such that $UI_n \preceq \Hf_i(x) \preceq LI_n,$ for $i\in [m].$ \\
	\indent  \textit{(iii)} $\Hf_i(x)$ is Lipschitz continuous with constant $M$, $i\in [m].$ \\
	\indent  \textit{(iv)} For any $x\in \mathbb{R}^{n}.$ there exists constant $a$ and $b$ such that
	$$aI_n \preceq B(x) \preceq bI_n.$$
\end{assumption}
The Assumption \ref{assum3} $(ii)$ implies that  $f_i(x)$ is strongly convex  with modulus $U$ and $\Hf_i(x)$ is Lipschitz continuous with constant $L$, for $i\in [m]$.

Next, we investigate the R-linear convergence of Algorithm \ref{algo1}. 

\begin{lemma}\label{lem:R-line}
	Suppose that Assumption \ref{assum1} and Assumption \ref{assum3} hold. Let $\{x_k\}$ be the sequence generated by Algorithm \ref{algo1}. Then, for any $k\geqslant 0$,\\
	\textit{(i)}\quad $\|x_k - x^*\| \leqslant \frac{2}{U}\|d_{SD}^k\|,$\\
	\textit{(ii)}\quad $\|s_k\| \geqslant \frac{1-\sigma_2}{L}\cos\beta_k\|d_{SD}^k\|.$
	
\end{lemma}
\begin{proof} 
	The two assertions can be obtained by using the same arguments as in the proof of \cite[Lemma 4.4]{wolf-quasi-Newton}, we omit them here.
\end{proof}
\begin{proposition}
	Let $\left\{ {{x}_{k}} \right\}$ be the sequence generated by Algorithm \ref{algo1}. Suppose that  Assumption \ref{assum1} and Assumption \ref{assum3} hold. Then, there exists a constant $r$ such that
	\begin{equation}\label{bizhi_ 1}
		\frac{\gamma_k^{\top}s_k}{\|s_k\|^{2}} \geqslant U ,
	\end{equation}
	\begin{equation}\label{bizhi_ 2}
		\frac{\|\gamma_k\|^{2}}{\gamma_k^{\top}s_k} \leqslant \frac{r^2}{U}.
	\end{equation}
\end{proposition}
\begin{proof}
	By the definition of $\gamma_k$, $s_k$ and $m_k > 0$, we have 
	\begin{equation}
		\frac{\gamma_k^{\top}s_k}{\|s_k\|^2} = \frac{(y_k + m_ks_k )^{\top}s_k}{\|s_k\|^2} = \frac{y_k^{\top}s_k}{\|s_k\|^2} +m_k\geqslant U.
	\end{equation}
	Since each objective function is U-strongly convex, for  $i \in [m]$ and any $x\in {{\mathbb{R}}^{n}}$, ${{f}_{i}}(x)$ has a lower bound $f_{i}^{\min }.$ \\
	It follows from the proof of Proposition \ref{pro:dot-bizhi} that 
	$$\|\gamma_k\| \leqslant   r\|s_k\|,$$
	where $ r = 2\max\limits_{i\in [m]} L_i + \max\limits_{i\in [m]}(f_i(x_0)-f_{i}^{min})$, the proof is complete.
	
\end{proof}
\begin{theorem}\label{them:R-line convergence}
	Suppose that Assumption \ref{assum1} and Assumption \ref{assum3} hold. Let  $\{x_k\}$ be the sequence generated by Algorithm \ref{algo1} and $x^*$ is defined as Theorem \ref{them:convex_global}. Then,\\
	\textit{(i)}\quad$\{x_k\}$  converges R-linearly to  $x^*,$  \\
	\textit{(ii)}\quad $\sum\limits_{k\geqslant 0} \|x_k-x^*\| < \infty.$
	
\end{theorem}

\begin{proof}
	The proof of the theorem is similar to the corresponding discussion in \cite[Proposition 4.1]{limited BFGS}, so we omit the proof here. 
	
\end{proof}

\begin{remark}\label{rem:MFONMO and MQNMO}
	Using the convexity of $f(x)$, Proposition \ref{pro:step_bound} and Assumption \ref{assum3}(\textit{iv}), we can get
	
	\begin{align}
		0 \leqslant  \sum\limits_{i\in [m]}\lambda_{i}^k(f_i(x_k)-f_i(x_{k+1})) 
		& \leqslant  (\sum\limits_{i\in [m]}\lambda_{i}^k\nabla f_{i}(x_k))^{\top}\left(x_k - x_{k+1} \right) \nonumber\\
		& = d_k^{\top}B_ks_k  \nonumber\\
		& \leqslant  \frac{b}{\alpha_1}\|s_k\|^2  \ \ \ \ \forall k\geqslant  0,\label{DOC}
	\end{align}
	where $\lambda_k={{\left( \lambda_1^k,\lambda_2^k, 
			\cdots, \lambda_m^k\right)}^\top}$is the solution of \eqref{MDP}.  In the convex case, we can observe that $\eta_k >0$ and by Theorem \ref{them:convex_global} get
	$$ \|s_k\| = \|x_{k+1}-x_k\| \to 0, \ k \to \infty.$$
	Then,
	$$\sum\limits_{i\in [m]}\lambda_{i}^k(f_i(x_k)-f_i(x_{k+1})) \to 0, \ k \to \infty,$$
	i.e., 
	\begin{align*}
		m_k &= \max \{-\eta_k, 0\}  + \sum\limits_{i\in [m]} \lambda_i^k (f_i(x_k)-f_i(x_{k+1})) \\
		&= \sum\limits_{i\in [m]} \lambda_i^k (f_i(x_k)-f_i(x_{k+1}))\to 0,\ k\to \infty.
	\end{align*}
	in \eqref{gammak}. Based on the analysis results above, for a sufficiently large $k$, $B_{k+1}s_k \approx y_k$. Then \eqref{section equation} can be seen as an approximation of the secant equation in \cite{modified}. Remarkably, for some convex functions, both MFQNMO and MQNMO in \cite{modified} may exhibit comparable convergence rates. This observation can be supported by the numerical experiments presented in the following section. It can be seen that MFQNMO not only ensure the effectiveness of the algorithm for convex problems, but also guarantee the global convergence for nonconvex MOPs.
\end{remark}

Now, we show some useful tools that will be used in superlinear convergence analysis. These tools help to characterize the Dennis-Mor$\acute{e}$ condition and demonstrate the convergence rate.\\
\indent Denote the average matrix $\bar{G}_{i}^k$ as follows
$$ \bar{G}_{i}^k = \int_{0}^{1}\Hf_{i}(x_k+\tau s_k)d\tau,\quad i\in [m].$$
Then,
\begin{equation}\label{{G}ikks_k}
	\bar{G}_{i}^k s_k = \mu_i^k.
\end{equation}
Next, we define another two auxiliary functions as follows 
$$f_{\lambda}^{k}(x) = \sum\limits_{i\in [m]}\lambda_{i}^kf_i(x),$$
$$G_{\lambda}^k = \int_{0}^{1}\sum\limits_{i\in [m]}\lambda_{i}^k \Hf_{i}(x_k+\tau s_k)d\tau.$$
where $\lambda_k={{\left( \lambda_1^k,\lambda_2^k, 
		\cdots, \lambda_m^k\right)}^\top}$is the solution of \eqref{MDP}. 
Let us define
$$\bar{s}_k = \Hf_{\lambda}^{k}(x^*)^{\frac{1}{2}}s_k, \quad \bar{y}_k = \Hf_{\lambda}^{k}(x^*)^{-\frac{1}{2}}y_k, \quad \bar{\gamma}_k = \Hf_{\lambda}^{k}(x^*)^{-\frac{1}{2}}\gamma_k,$$
\begin{equation}\label{bar{B}_k}
	\bar{B}_k = \Hf_{\lambda}^{k}(x^*)^{-\frac{1}{2}}B_k\Hf_{\lambda}^{k}(x^*)^{-\frac{1}{2}},
\end{equation}
$$\cos\bar{\beta}_k = \frac{\bar{s}_k^{\top}\bar{B}_k\bar{s}_k}{\|\bar{s}_k\|\|\bar{B}_k\bar{s}_k\|},\quad \bar{q}_k =\frac{\bar{s}_k^{\top}\bar{B}_k\bar{s}_k}{\bar{s}_k^{\top}\bar{s}_k},\quad \bar{a}_k = \frac{\bar{\gamma}_k^{\top}\bar{s}_k}{\|\bar{s}_k\|^2},\quad \bar{b}_k = \frac{\|\bar{\gamma}_k\|^2}{\bar{\gamma}_k^{\top}\bar{s}_k}.$$

By the BFGS-tpye matrix updated fomula \eqref{MFBFGS}, it can be noted that \eqref{bar{B}_k} leads to the formula
$$\bar{B}_{k+1} = \bar{B}_{k}-\dfrac{\bar{B}_{k}\bar{s}_k\bar{s}_k^\top \bar{B}_{k}}{\bar{s}_k^\top \bar{B}_k\bar{s}_k} +\dfrac{\bar{\gamma}_k\bar{\gamma}_k^\top }{\bar{\gamma}_{k}^\top \bar{s}_k} .$$
Another useful tool used in \cite{tool}, given that the following function 
$$\psi(B) =\rm trace(B)-\ln(\det(B)),$$
\begin{equation}\label{bar{B}_{k+1}}
	\psi(\bar{B}_{k+1}) = \psi(\bar{B}_{k}) +(\bar{b}_k-\ln(\bar{a}_k
	)-1)+ \ln(\cos^2\bar{\beta}_k) + \left[1 - \frac{\bar{q}_k}{\cos^2\bar{\beta_k}} + \ln(\frac{\bar{q}_k}{\cos^2\bar{\beta_k}})\right]
\end{equation}

In the next theorem, we prove that the Dennis-Mor$\acute{e}$ condition of superlinear convergence holds, a similar result is discussed in \cite{Global quasi-Newton} and \cite{tool}.  

\begin{theorem}
	Suppose that Assumption \ref{assum1} and Assumption \ref{assum3} hold. Let $\{x_k\}$ be the sequence generated by Algorithm \ref{algo1}. Then
	$$
	\lim\limits_{k \to \infty} \frac{\|(B_k - \Hf_{\lambda}^{k}(x^*))s_k\|}{\|s_k\|}  = 0.
	$$
\end{theorem}

\begin{proof}
	By \eqref{{G}ikks_k}, we can get
	$$\bar{G}_{\lambda}^ks_k = y_k,$$
	hence,
	\begin{equation}\label{bar{G}_{i}k}
		(\bar{G}_{\lambda}^k-\Hf_{\lambda}^{k}(x^*))s_k = y_k - \Hf_{\lambda}^{k}(x^*)s_k, \quad k\geqslant  0.  
	\end{equation}
	
	It follows by Assumption \ref{assum3} \textit{(iii)}, \eqref{bar{G}_{i}k} and the definition of $\bar{y}_k$, $\bar{s}_k$ that
	\begin{align}
		\|\bar{y}_k - \bar{s}_k\|
		& = \|\Hf_{\lambda}^{k}(x^*)^{-\frac{1}{2}} (\bar{G}_{\lambda}^k-\Hf_{\lambda}^{k}(x^*))\Hf_{\lambda}^{k}(x^*)^{-\frac{1}{2}}\bar{s}_k \| \nonumber \\
		& \leqslant  \|\Hf_{\lambda}^{k}(x^*)^{-\frac{1}{2}}\|^2 \|\bar{s}_k\| \|\bar{G}_{\lambda}^k-\Hf_{\lambda}^{k}(x^*)\| \nonumber \\
		&\leqslant  M\|\Hf_{\lambda}^{k}(x^*)^{-\frac{1}{2}}\|^2 \|\bar{s}_k\| \int_{0}^{1}\|x_k + \tau s_k - x^*\|d\tau \nonumber \\
		& \leqslant  \bar{c} \|\bar{s}_k\| \varepsilon_k , \label{bar{y}_k - bar{s}_k}
	\end{align}
	where $\bar{c} = M\|\Hf_{\lambda}^{k}(x^*)^{-\frac{1}{2}}\|^2$, $\varepsilon_k = \max\{\|x_{k+1} - x^*\|, \|x_k - x^*\|\}.$
	Since 
	$$|\|\bar{y}_k\| - \|\bar{s}_k \|| \leqslant  \|\bar{y}_k - \bar{s}_k \|,$$
	by \eqref{bar{y}_k - bar{s}_k}, we get
	\begin{equation}\label{bar{y}_k}
		(1-\bar{c}\varepsilon_k)\|\bar{s}_k\| \leqslant  \|\bar{y}_k\| \leqslant  (1+\bar{c}\varepsilon_k)\|\bar{s}_k\|,
	\end{equation}
	$$(1-\bar{c}\varepsilon_k)^2\|\bar{s}_k\|^2 -2\bar{y}_k^{\top}\bar{s}_k +\|\bar{s}_k\|^2 \leqslant  \|\bar{y}_k\|^2 -2\bar{y}_k^{\top}\bar{s}_k +\|\bar{s}_k\|^2 \leqslant  \bar{c}^2 \|\bar{s}_k\|^2 \varepsilon_k^2.$$
	It follows that
	\begin{equation}\label{bar{y}_k^{top} bar{s}_k}
		\frac{\bar{y}_k^{\top}\bar{s}_k}{\|\bar{s}_k\|^2} \geqslant  1-\bar{c}\varepsilon_k.
	\end{equation}
	Using the Assumption \ref{assum3} \textit{(ii)}, and \eqref{bar{y}_k^{top} bar{s}_k} we obtain
	\begin{align}
		\bar{a}_k = \frac{\bar{\gamma}_k^{\top}\bar{s}_k}{\|\bar{s}_k\|^2} =  \frac{(\bar{y}_k + m_k\Hf_{\lambda}^{k}(x^*)^{-1}\bar{s}_k )^{\top}\bar{s}_k}{\|\bar{s}_k\|^2} 
		& = \frac{\bar{y}_k^{\top}\bar{s}_k}{\|\bar{s}_k\|^2} + m_k\frac{\bar{s}_k^{\top} \Hf_{\lambda}^{k}(x^*)^{-1}\bar{s}_k}{\|\bar{s}_k\|^2} \nonumber \\
		& \geqslant  1-\bar{c}\varepsilon_k +\frac{m_k}{L} \nonumber \\
		& \geqslant 1-\bar{c}\varepsilon_k. \label{bar{a}_k}
	\end{align}
	By the L-Lipschitz continuity of $f_i$, we have 
	\begin{equation}\label{U}
		\frac{\|y_k\|}{\|s_k\|} \leqslant L.
	\end{equation}
	From \eqref{bizhi_ 1}, \eqref{U} and the right side of  \eqref{bar{y}_k}, 
	\begin{align}
		\bar{b}_k 
		& =  \frac{\|\bar{\gamma}_k\|^2}{\bar{\gamma}_k^{\top}\bar{s}_k} \nonumber \\
		& = \frac{\|\bar{y}_k\|^2}{\bar{\gamma}_k^{\top}\bar{s}_k} + 2m_k\frac{\bar{y}_k^{\top} \Hf_{\lambda}^{k}(x^*)^{-1}\bar{s}_k}{\bar{\gamma}_k^{\top}\bar{s}_k} + (m_k)^2\frac{\bar{s}_k^{\top} \Hf_{\lambda}^{k}(x^*)^{-2}\bar{s}_k}{\bar{\gamma}_k^{\top}\bar{s}_k} \nonumber \\
		& \leqslant \frac{(1+\bar{c}\varepsilon_k)^2}{1-\bar{c}\varepsilon_k} 
		+ 2m_k\frac{\|\Hf_{\lambda}^{k}(x^*)^{-\frac{1}{2}}\| \|y_k\| \|\Hf_{\lambda}^{k}(x^*)^{-1}\| \|\Hf_{\lambda}^{k}(x^*)^{\frac{1}{2}}\| \|s_k\|}{U\|s_k\|^2}  + (m_k)^2\frac{\|\Hf_{\lambda}^{k}(x^*)^{-2}\| \|\bar{s}_k\|^2}{U\|s_k\|^2} \nonumber \\
		& \leqslant 1+\frac{3\bar{c}+ \bar{c}^2\varepsilon_k}{1 - \bar{c}\varepsilon_k}\varepsilon_k + \eta_1 m_k +\eta_2 m_k^2 ,
	\end{align}
	where $\eta_1 = 2\frac{\|\ \Hf_{\lambda}^{k}(x^*)^{-\frac{1}{2}}\| \|\Hf_{\lambda}^{k}(x^*)^{-1}\| \|\Hf_{\lambda}(x^*)^{\frac{1}{2}}\| L}{U}$, $\eta_2 = \frac{\|\Hf_{\lambda}^{k}(x^*)^{-2}\| \|\Hf_{\lambda}^{k}(x^*)^{\frac{1}{2}}\|^2}{U}.$
	Using \eqref{bar{B}_{k+1}} and \eqref{bar{a}_k},
	\begin{align*}
		\psi(\bar{B}_{k+1}) 
		\leqslant \psi(\bar{B}_{k}) 
		& +\left(\frac{3\bar{c}+ \bar{c}^2\varepsilon_k}{1 - \bar{c}\varepsilon_k}\varepsilon_k + \eta_1 m_k +\eta_2 m_k^2 -\ln(1-\bar{c}\varepsilon_k)\right)\\
		& + \ln(\cos^2\bar{\beta}_k) +\left[1 -\frac{\bar{q}_k}{\cos^2\bar{\beta_k}} + \ln\left(\frac{\bar{q}_k}{\cos^2\bar{\beta_k}}\right)\right],
	\end{align*}
	By adding the above inequality from $1$ to $k$, then we obtain
	
	\begin{align}
		0< \psi(\bar{B}_{k+1}) 
		\leqslant \psi(\bar{B}_{1}) 
		& +\sum_{j=1}^k\left(\frac{3\bar{c}+ \bar{c}^2\varepsilon_j}{1 - \bar{c}\varepsilon_j}\varepsilon_j + \eta_1 m_j +\eta_2 m_j^2 -\ln(1-\bar{c}\varepsilon_j)\right) \nonumber\\
		& + \left[1 -\frac{\bar{q}_j}{\cos^2\bar{\beta_j}} + \ln\left(\frac{\bar{q}_j}{\cos^2\bar{\beta_j}}\right)\right] +\ln \left(\cos^2\bar{\beta}_j\right),\label{psi(bar{B}{k+1})}
	\end{align}
	From Theorem \ref{them:R-line convergence}, there exists $k_0 \in \mathbb{N}$ such that
	$$\bar{c}\varepsilon_k < \frac{1}{2}, \ \forall k > k_0,$$
	then, 
	\begin{equation}\label{varepsilon_k}
		\ln(1-\bar{c}\varepsilon_k) \geqslant -2\bar{c}\varepsilon_k.
	\end{equation}
	We can get
	\begin{equation}\label{m_k}
		\sum\limits_{k \geqslant 1}m_k < \infty,
	\end{equation} 
	from \eqref{DOC} and Theorem \ref{them:R-line convergence}.
	It follows by \eqref{varepsilon_k} and \eqref{m_k} that
	\begin{align}
		\sum\limits_{j>k_0}\left(\frac{3\bar{c}+ \bar{c}^2\varepsilon_j}{1 - \bar{c}\varepsilon_j}\varepsilon_j + \eta_1 m_j +\eta_2 m_j^2 - \ln(1 - \bar{c}\varepsilon_j)\right) 
		& < \sum\limits_{j>k_0} 9\bar{c}\varepsilon_j + \sum\limits_{j>k_0} (\eta_1 m_j +\eta_2 m_j^2) \nonumber\\
		& < \sum_{j=1}^{\infty}9\bar{c}\varepsilon_j + \sum_{j=1}^{\infty}(\eta_1 m_j +\eta_2 m_j^2)\nonumber\\
		& < \infty. \label{Series convergence}
	\end{align}
	By \eqref{Series convergence} and \eqref{psi(bar{B}{k+1})}, we get
	$$\sum_{j=1}^{\infty}\left(-\left[1 -\frac{\bar{q}_j}{\cos^2\bar{\beta_j}} + \ln\left(\frac{\bar{q}_j}{\cos^2\bar{\beta_j}}\right)\right] -\ln \left(\cos^2\bar{\beta}_j\right) \right) < \infty.$$
	
	\noindent Since $-\ln(\cos^2\bar{\beta}_j)>0$ ,for all $j>k_0$, and $h(t) = 1-t+lnt \leqslant 0$ when $t>0$, hence, we get
	$$ -\left[1 -\frac{\bar{q}_j}{\cos^2\bar{\beta_j}} + \ln\left(\frac{\bar{q}_j}{\cos^2\bar{\beta_j}}\right)\right] \geqslant 0,$$
	$$\lim\limits_{k\to \infty}\ln \left(\frac{1}{\cos^2\bar{\beta}_k}\right)  = 0,\quad \lim\limits_{k\to \infty}\left[1 -\frac{\bar{q}_k}{\cos^2\bar{\beta_k}} + \ln\left(\frac{\bar{q}_k}{\cos^2\bar{\beta_k}}\right)\right] = 0.$$
	The above relations imply that
	$$\lim\limits_{k\to \infty}\ln \left( \cos^2\bar{\beta}_j \right)  = 1,\quad \lim\limits_{k\to \infty}\bar{q}_k = 1.$$
	Now,
	\begin{align*}
		\lim\limits_{k\to \infty}\frac{\|(B_k - \Hf_{\lambda}^{k}(x^*))s_k\|^2}{\|s_k\|^2}
		& = \lim\limits_{k\to \infty}\frac{\|\Hf_{\lambda}^{k}(x^*)^{-\frac{1}{2}}(B_k - \Hf_{\lambda}^{k}(x^*))\Hf_{\lambda}^{k}(x^*)^{-\frac{1}{2}}\bar{s}_k\|^2}{\|\Hf_{\lambda}^{k}(x^*)^{-\frac{1}{2}}s_k\|^2} \\
		& =  \lim\limits_{k\to \infty}\frac{\|(\bar{B}_k - I_n)\bar{s}_k\|^2}{\|\bar{s}_k\|^2} \\
		& = \lim\limits_{k\to \infty}\frac{\|\bar{B}_k\bar{s}_k\|^2  - 2\bar{s}_k^{\top}\bar{B}_k\bar{s}_k + \|\bar{s}_k\|^2} {\|\bar{s}_k\|^2} \\
		& = \lim\limits_{k\to \infty}\left[\frac{\bar{q}_k^{2}}{\cos^2\bar{\beta_k}} - 2\bar{q}_k +1\right] \\
		& = 0
	\end{align*}
	This completes the proof.
\end{proof}

\begin{theorem}
	Suppose that Assumption \ref{assum1} and Assumption \ref{assum3} hold. Let $\{x_k\}$ be the sequence generated by Algorithm \ref{algo1}. Then, \\
	\indent (i)\ \eqref{wolf1} and \eqref{wolf2} hold with $\alpha_k=1$ for sufficiently large $k$.\\
	\indent (ii)\ $\{x_k\}$ converges to $x^*$ at a superlinear rate. 
\end{theorem}

\begin{proof}
	The proof is similar to that in \cite[Theorem 4.3]{Global quasi-Newton}, we omit it here. 
\end{proof}

\section{Numerical experiments} \label{sec5:numerical}

\indent In this section, we present some numerical experiments to verify the efficiency of MFQNMO. We compare MFQNMO with QNMO \cite{quasi-Newton_1} with Wolfe line search and MQNMO \cite{modified} with Wolfe line search. The selected test problems include convex and nonconvex problems. Our evaluation criteria include the method's effectiveness in finding optimal solutions and its ability to approximate real Pareto surfaces accurately. These experiments were implemented using Python 3.9.\\
\indent For MQNMO and MFQNMO methods, we used the Frank-Wolf gradient method to solve the dual problem \eqref{MDP}.  For QNMO, \eqref{quasi-newton} is solved using scipy.optimize, which is an optimization Solver. All test problems and main characteristics are listed in Table \ref{tab1}. Columns `n' and `m' indicate the number of variables and objectives, respectively. The `Convex' column indicates whether the problem is convex or not. `Y' stands for convex MOP. `N' stands for nonconvex MOP.  Problems are solved under box constraints ${{x}_{L}}\leqslant x\leqslant {{x}_{U}}.$ The initial points are randomly selected within these box constraints ${{x}_{L}}\leqslant x\leqslant {{x}_{U}}.$ We used the termination criterion  $\left| \theta \left( x \right) \right|\leqslant {{10}^{-8}}$ for all test problems. The maximum number of iterations is set to 500. To obtain the stepsize that satisfies the Wolfe conditions \eqref{wolf1} and \eqref{wolf2}, we used the Wolf line search algorithm proposed in \cite{limited BFGS}. We used $\sigma_1 = {10}^{-4}$, $\sigma_2 = 0.1$, and set $B_0 = I_n.$ We perform 200 computations using the same initial points for all methods. 
\begin{table}
	\caption{The description of all problems used in numerical experiments} \label{tab1}
	\centering
	\begin{tabular}{lllllll}
		\toprule
		problem & $n$ & $m$ & ${{x}_{L}}$ & ${{x}_{U}}$ & convex & Reference \\
		\midrule
		SD & 4 & 2 & {\( (1,-\sqrt{2},-\sqrt{2},1) \)} & {\( (3,3,3,3) \)} & Y & \cite{SD} \\
		PNR & 2 & 2 & {\( (-2,-2) \)} & {\( (2,2) \)} & Y & \cite{PNR} \\
		JOS1a & 50 & 2 & \( (-2,\cdots,-2) \) & \( (2,\cdots,2) \) & Y & \cite{JOS} \\
		JOS1b & 100 & 2 & \( (-2,\cdots,-2) \) & \( (2,\cdots,2) \) & Y & \cite{JOS} \\
		DGO1 & 1 & 2 & \( -10 \) & \( 13 \) & N & \cite{test review} \\
		DGO2 & 1 & 2 & \( -9 \) & \( 9 \) & Y & \cite{test review} \\
		Lov1 & 2 & 2 & \( (-10,-10) \) & \( (10,10)\) & Y &  \cite{LOV} \\
		Lov2 & 2 & 2 & \( (-0.75,-0.75) \) & \( (0.75,0.75) \) & N &  \cite{LOV} \\
		Lov3 & 2 & 2 & \( (-20,-20) \) & \( (20,20) \) & N &  \cite{LOV} \\
		Lov4 & 2 & 2 & \( (-20,-20) \) & \( (20,20) \) & N &  \cite{LOV} \\
		SK1 & 1 & 2 & \( -100 \) & \( 100 \) & N &  \cite{test review} \\
		BK1 & 2 & 2 & \( (-5,-5) \) & \( (10,10) \) & Y &  \cite{test review} \\
		SLCDT1 & 2 & 2 & \( (-1.5,-1.5) \) & \( (1.5,1.5) \) & N &  \cite{SLCDT} \\
		MOP1 & 1 & 2 & \( -{10}^{5} \) & \( {10}^{5} \) & Y & \cite{test review} \\
		MOP2 & 2 & 2 & \( (-4,-4) \) & \( (4,4) \) & N & \cite{test review} \\
		LDTZ & 3 & 3 & \( (0,0,0) \) & \( (1,1,1) \) & N & \cite{LDTZ} \\
		Hil1 & 2 & 2 & \( (0,0) \) & \( (1,1) \) & N & \cite{Hil} \\
		AP2 & 1 & 2 & \( -100 \) & \( 100 \) & Y & \cite{modified} \\
		AP3 & 2 & 2 & \( (-100,-100)) \) & \( (100,100) \) & N & \cite{modified} \\
		FF1 & 2 & 2 & \( (-1,-1) \) & \( (1,1) \) & N & \cite{test review} \\
		KW2 & 2 & 2 & \( (-3,-3) \) & \( (3,3) \) & N & \cite{KW2} \\
		MHHM1 & 1 & 3 & \( 0 \) & \( 1 \) & Y & \cite{test review} \\
		MHHM2 & 2 & 3 & \( (0,0) \) & \( (1,1) \) & Y & \cite{test review} \\
		\bottomrule
	\end{tabular}
\end{table}

\begin{table}
	\caption{Number of average iterations (iter), average CPU time (time(ms)), number of average function evaluations (feval), and number of failure points (NF) of MFQNMO, QNMO, and MQNMO.}\label{tab2}
	\centering
	\resizebox{\textwidth}{!}{
	\begin{tabular}{*{14}{c}}
		\toprule
		\multirow{2}*{Problem} & \multicolumn{4}{c}{QNMO} & \multicolumn{4}{c}{MFQNMO} & \multicolumn{4}{c}{MQNMO} \\
		\cmidrule(lr){2-5} \cmidrule(lr){6-9} \cmidrule(lr){10-13}
		& {iter} & {time} & {feval} & {NF} & {iter} & {time} & {feval} & {NF}& {iter} & {time} & {feval} & {NF} \\
		\midrule
		SD & 13.49 & 486.47 & 15.54 & \textbf{1} &\textbf{7.35} & \textbf{2.55} & \textbf{18.51}& \textbf{1} & 11.17 & 5.53 & 43.54 & \textbf{1} \\
		
		PNR & \textbf{2.56} & 118.66 & \textbf{5.48} & \textbf{0} & 7.87 & \textbf{2.81} & 28.31 & \textbf{0} & 8.12& 3.84 & 54.67 & \textbf{0} \\
		
		JOS1a & \textbf{3.01} & 270.14 & \textbf{7.39}& \textbf{0} & 5.96 & \textbf{3.84}& 32.89 & \textbf{0} & 3.05 & 5.38 & 24.51 & \textbf{0} \\
		
		JOS1b & \textbf{2.89} & 12725.46 & \textbf{7.39} & \textbf{0} & 7.65 &\textbf{19.38} & 10.05 & \textbf{0}& 11.23 & 43.337 & 103.01 & \textbf{0} \\
		
		DGO1 & 175.51 & 1906.25 & 5058.35 & 70 & \textbf{1.39}	&\textbf{0.64} &\textbf{4.14}	&\textbf{0} & 1.44 &0.72 & 4.37 &\textbf{0}	 \\
		
		DGO2 & \textbf{3.59} & 52.34 & \textbf{9.56} & \textbf{0} & 5.40 & \textbf{1.64} & 17.41 & \textbf{0} & 70.35 & 31.75 & 409.96 & \textbf{0 }\\
		
		Lov1 & 268.91 & 3969.94	& 9213.41	& 107	& \textbf{4.27}	& \textbf{1.45}	& \textbf{12.72}	& \textbf{0} &4.30	&1.72 & 15.04 &\textbf{0}\\
		
		Lov2  & 30.43	&448.92	&77.89	&146 &\textbf{12.36}	&\textbf{8.41}	&\textbf{31.49}	& \textbf{124}	&29.70	&19.44	&107.27	&\textbf{124 }\\
		
		Lov3  & 78.77	&1462.87	&3400.81	&30	&\textbf{11.86}	&\textbf{6.27}	&\textbf{29.42}	&\textbf{3}	&24.13	&9.81	&71.69	&\textbf{3} \\
		
		Lov4 & 248.53 &4516.95	&12193.75	&101	&\textbf{1.30}	&\textbf{0.57}	&\textbf{4.04}	&\textbf{0}	&1.52	&0.71	&4.85	&\textbf{0} \\
		
		SK1 & 119.17 & 488.25	&2747.57	&69	&\textbf{2.36}	&\textbf{1.65}	&\textbf{26.54}	&\textbf{0} &2.66	&1.93	&46	&\textbf{0}\\
		
		BK1 & 98.41	&1085.38 &100.41 &39	&\textbf{1}	&\textbf{0.13}	&\textbf{3}	&\textbf{0}	&\textbf{1}	&0.16	&\textbf{3}	&\textbf{0} \\
		
		SLCDT1 & F	&F	&F	&200	&\textbf{3.35}	&\textbf{0.64}	&\textbf{8.51}	&\textbf{0}	&4.18	&1.38	&11.28	&\textbf{0} \\
		
		MOP1 & 4.66	&237.83	&146.13	&47	&1.09	&0.21	&4.89	&18	&\textbf{1}	&\textbf{0.15}	&\textbf{3}	&\textbf{0} \\
		
		MOP2 & 89.24 &1489.39	 &34344.11	 &35	 &\textbf{3.67}	 &\textbf{1.53}	 &\textbf{11.41}	 &\textbf{0}	 &7.18	 &2.28	 &121.99	 &1 \\
		
		LDTZ & F&F	&F	&200	&\textbf{18.02}	&\textbf{87.05}	&\textbf{119.06} &\textbf{18}	&F	&F	&F	&200 \\
		
		Hil1 & 258.88 &7409.41	&6274.12	&110	&\textbf{11.71}	&\textbf{6.41}	&\textbf{62.52}	&\textbf{16}	&98.39	&186.89	&529.98	&53\\
		
		AP2 & 1.04	&35.95	&3.17	&1	&\textbf{0.985}	&\textbf{0.41}	&\textbf{2.95}	&\textbf{0}	&\textbf{0.985}	&0.46	&\textbf{2.95}	&\textbf{0} \\
		
		AP3 & 147.91 &3363.81	&5353.84	&115	&\textbf{22.36}	&\textbf{10.41}	&\textbf{109.45} &\textbf{0}	&248.21	&204.68	&3858.64 &77 \\
		
		FF1 & 245.17	&3394.94	&10089.34	&98	&32.05	&15.97	&190.12	&\textbf{0}	&\textbf{16.12}	&\textbf{8.94}	&\textbf{81.77}	&\textbf{0} \\
		KW2& 92.97	&1697.45	&298.53	&103	&\textbf{13.81}	&\textbf{7.98}	&\textbf{51.01}	&\textbf{24}	&99.81	&69.37	&314.52	&32 \\
		MHHM1 & \textbf{0.88}	& 26.87	& \textbf{2.82}	& \textbf{0}	& \textbf{0.88}	& \textbf{0.40}	& \textbf{2.65}	& \textbf{0}	& \textbf{0.88}	& 0.42	& \textbf{2.65}	& \textbf{0} \\
		MHHM2 & 245.54	& 4752.67	& 9049.92	& 99	& \textbf{6.21}	& \textbf{43.96}	& \textbf{206.47}	& \textbf{31}	& 78.34	& 1356.45	& 3508.36	& 31 \\				
		\bottomrule
	\end{tabular}
		}
\end{table}
In Table \ref{tab2}, the average values obtained from these runs are recorded to determine the number of iterations (iter), function evaluations (feval), and CPU time (time(ms)). Among the 200 initial points, the number of points that did not converge in 500 iterations is recorded in column `NF' for MFQNMO, QNMO, and MQNMO. We denote `Failed' as `F,' which indicates that the corresponding method fails to solve the test problem under the stopping criterion. For large-scale problems, such as JOS, the QNMO method's CPU time consumption is substantial although it requires the fewest iterations. For all nonconvex test problems in Table \ref{tab1},  the QNMO method exhibited poor performance and the MQNMO method surpasses the MFQNMO method only for the FF1 test problem. For certain problems in Table \ref{tab1}, the experimental results of the MFQNMO method and the MQNMO method are equal approximately, such as BK1, MOP1, MHHM1 and AP2. We mentioned the possible reasons for this phenomenon in Remark \ref{rem:MFONMO and MQNMO}. As can be seen from Table \ref{tab2}, the MFQNMO method is more efficient and competitive.

\indent Figure \ref{fig1}-\ref{fig3} shows the Pareto frontier obtained using the MFQNMO method for some convex problems listed in Table \ref{tab1}. Figure \ref{fig4}-\ref{fig10} shows the Pareto frontier obtained using the MFQNMO method for certain nonconvex problems listed in Table \ref{tab1}. For the chosen test problems, it can be observed that the MFQNMO method was able to satisfactorily estimate the Pareto frontiers when considering evenly distributed starting points.

\begin{figure}[H]
	\centering
	\subfigure[QNMO]{
		\includegraphics[scale=0.32]{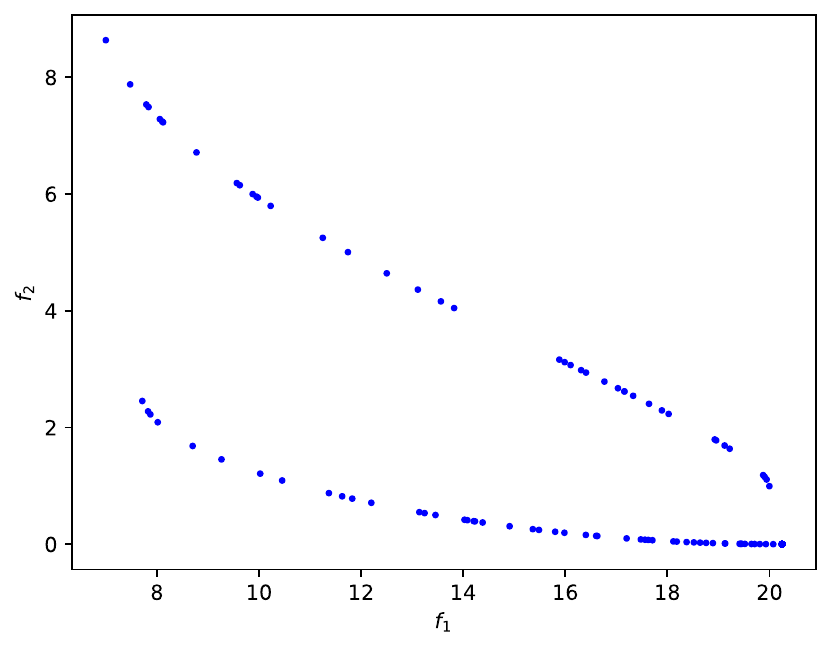} \label{1}
	}
	\quad
	\subfigure[MQNMO]{
		\includegraphics[scale=0.32]{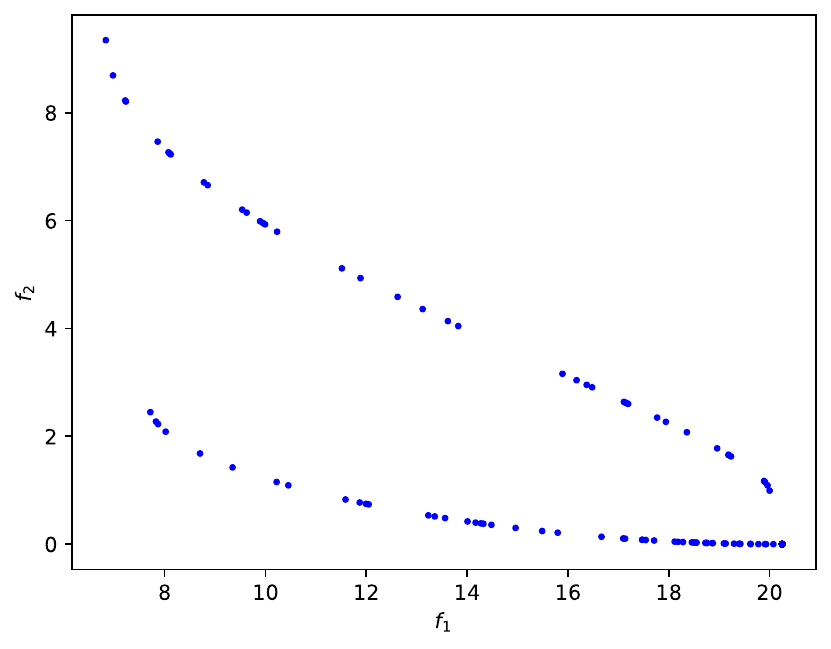}\label{2}
	}
	\quad
	\subfigure[MFQNMO]{
		\includegraphics[scale=0.32]{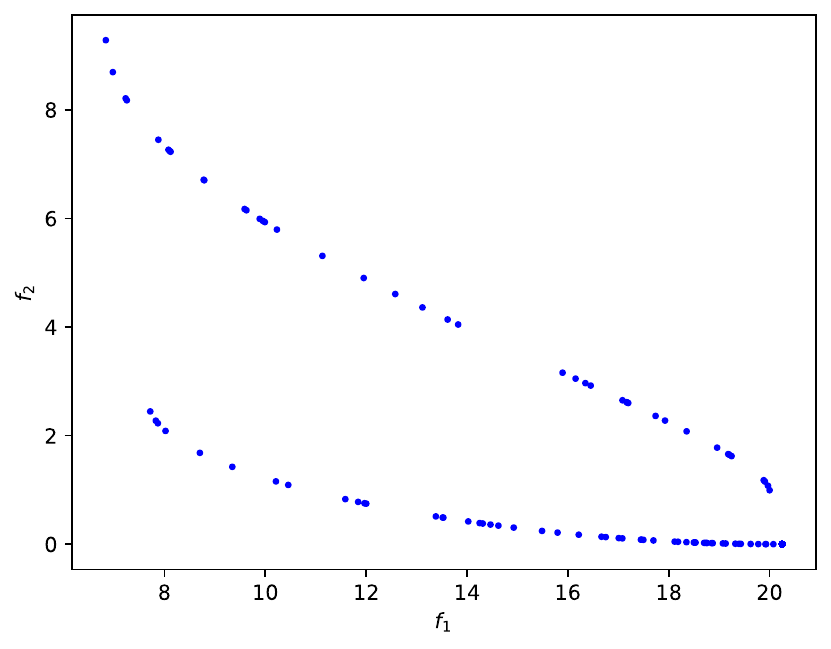}\label{3}
	}
	\quad
	\caption{The approximated nondominated frontiers generated by QNMO, MQNMO, and MFQNMO for the PNR problem.}  \label{fig1}
\end{figure}		

\begin{figure}[H]
	\centering
	\subfigure[QNMO]{
		\includegraphics[scale=0.32]{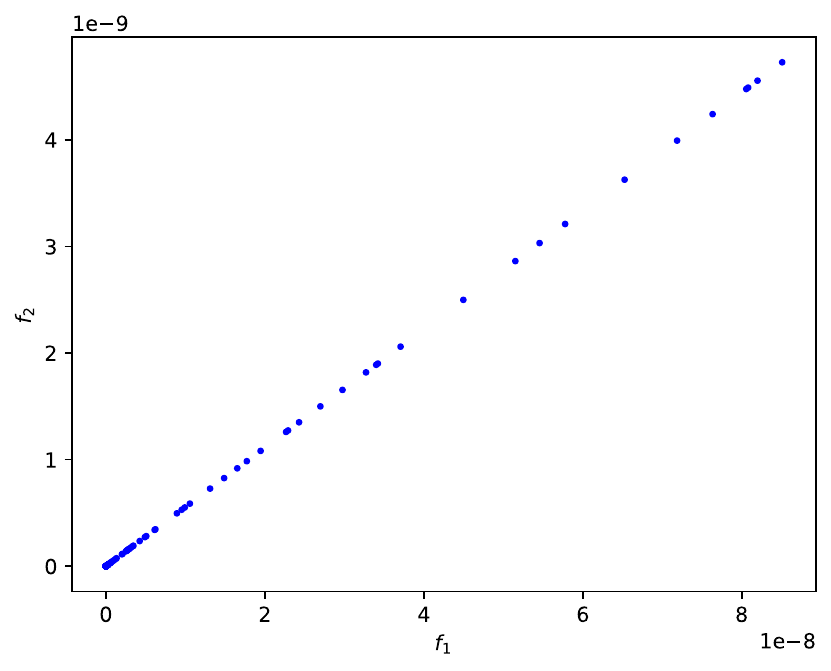} \label{1}
	}
	\quad
	\subfigure[MQNMO]{
		\includegraphics[scale=0.32]{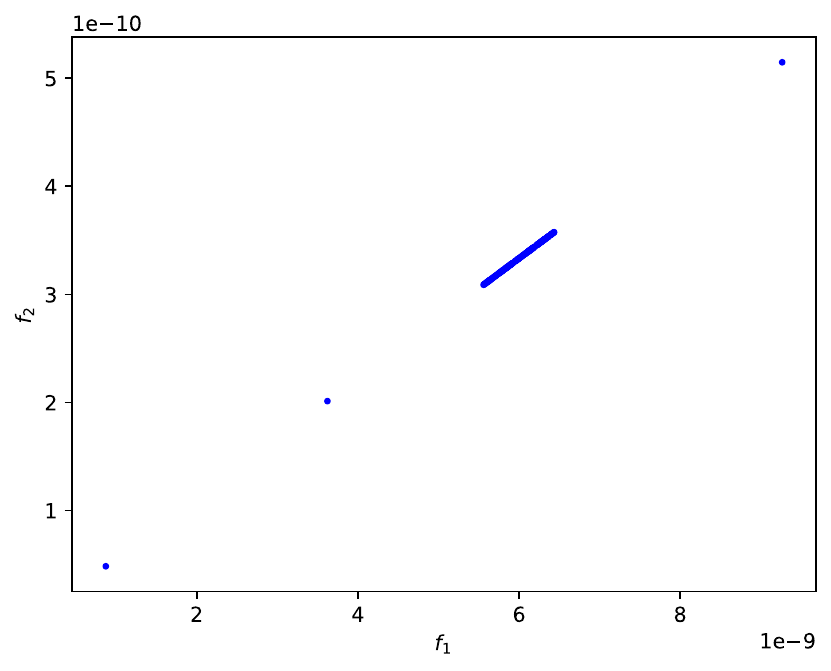}\label{2}
	}
	\quad
	\subfigure[MFQNMO]{
		\includegraphics[scale=0.32]{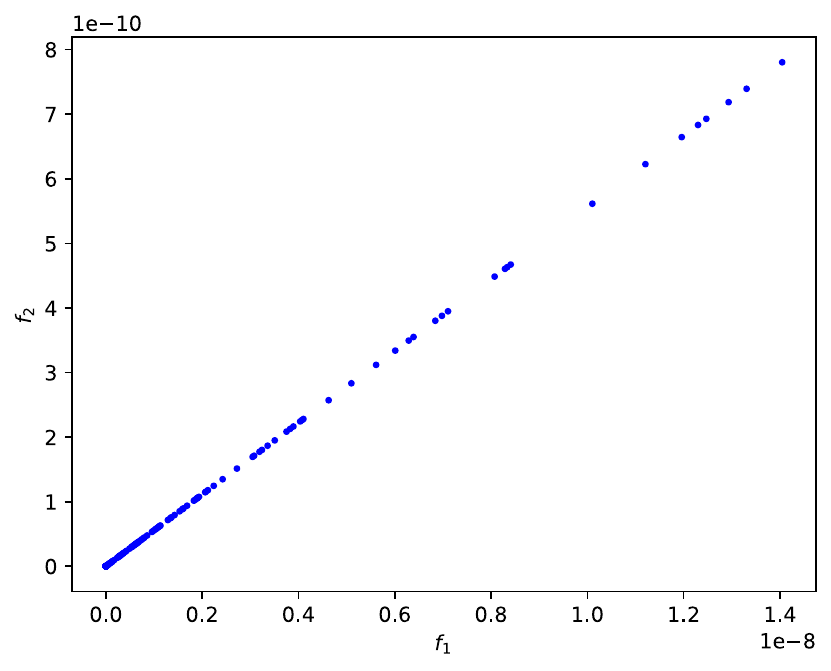}\label{3}
	}
	\quad
	\caption{The approximated nondominated frontiers generated by QNMO, MQNMO, and MFQNMO for the DGO2 problem.}  \label{fig2}
\end{figure}		

\begin{figure}[H]
	\centering
	\subfigure[QNMO]{
		\includegraphics[scale=0.32]{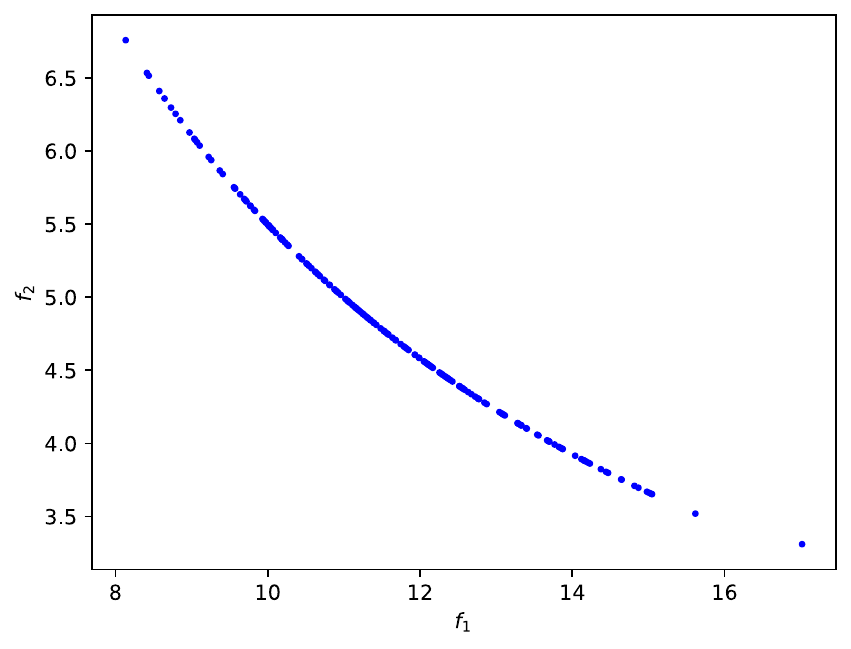} \label{1}
	}
	\quad
	\subfigure[MQNMO]{
		\includegraphics[scale=0.32]{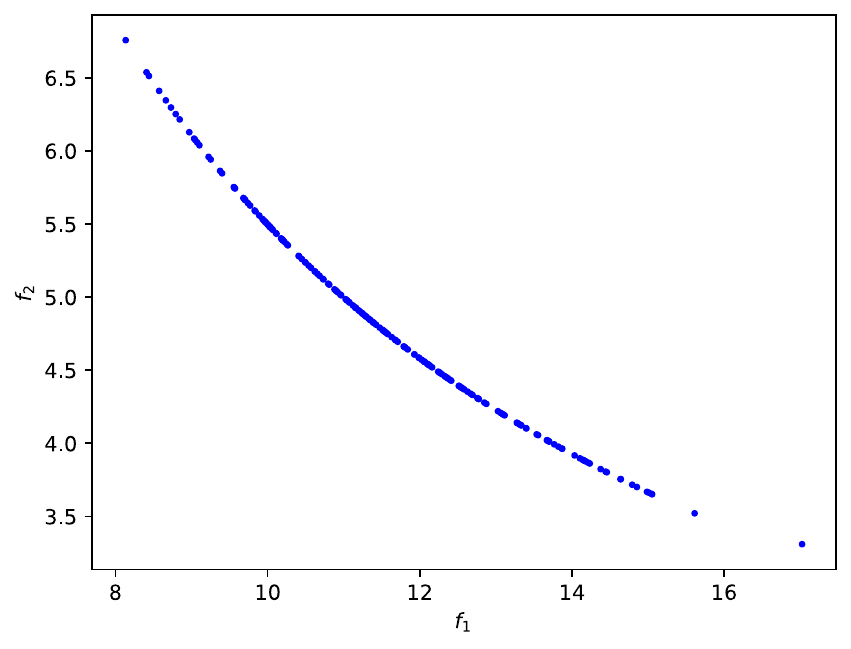}\label{2}
	}
	\quad
	\subfigure[MFQNMO]{
		\includegraphics[scale=0.32]{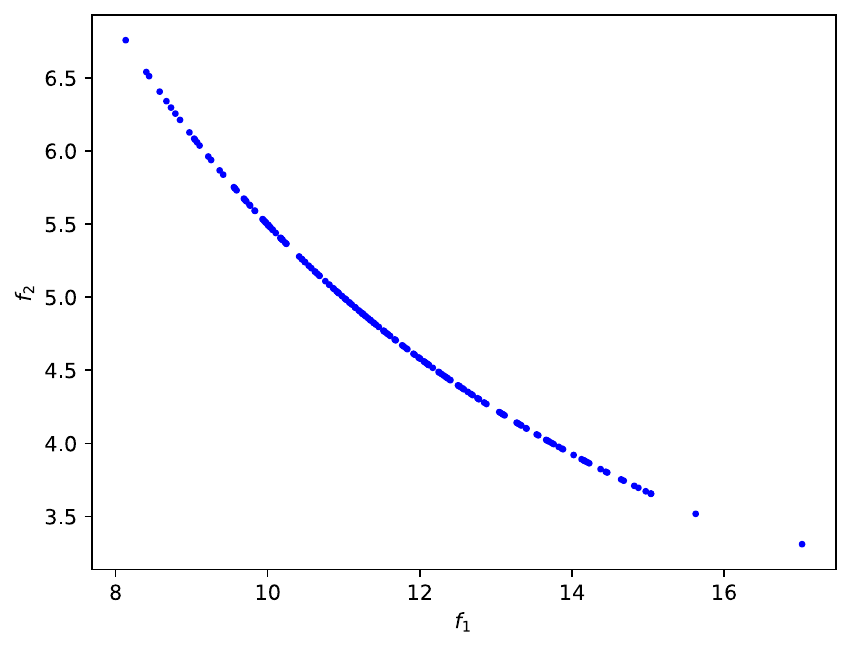}\label{3}
	}
	\quad
	\caption{The approximated nondominated frontiers generated by QNMO, MQNMO, and MFQNMO for the SD problem.}  \label{fig3}
\end{figure}

\begin{figure}[H]
	\centering
	\subfigure[QNMO]{
		\includegraphics[scale=0.32]{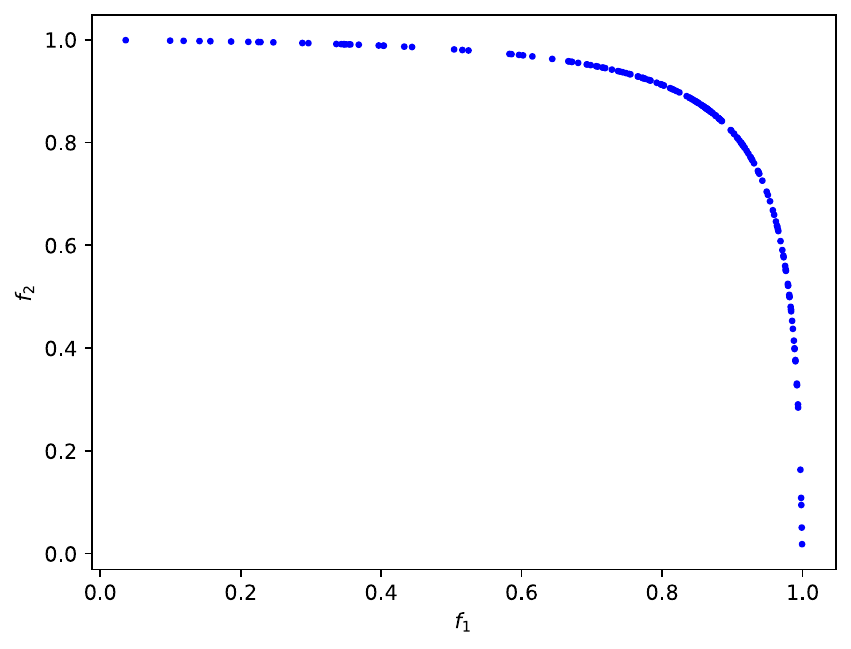} \label{1}
	}
	\quad
	\subfigure[MQNMO]{
		\includegraphics[scale=0.32]{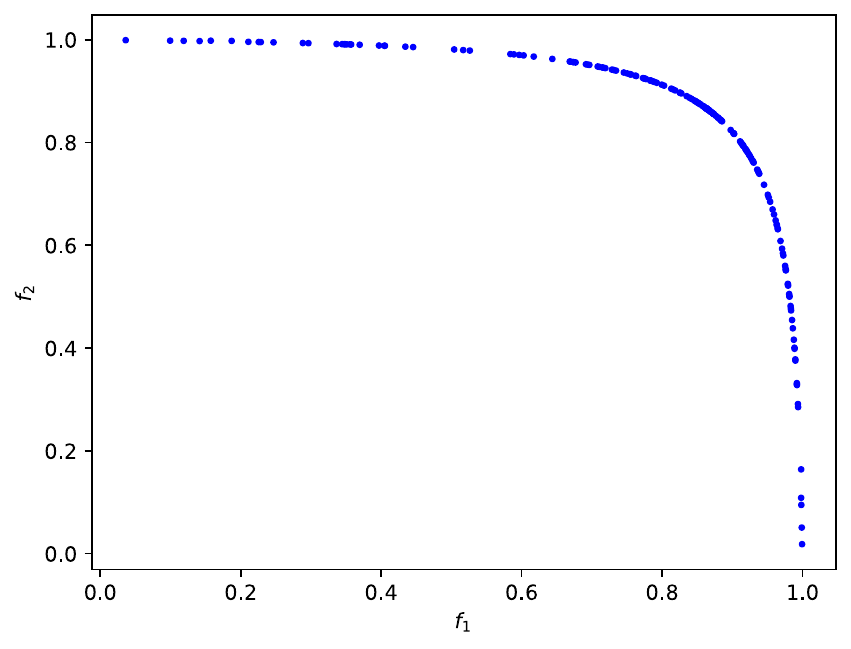}\label{2}
	}
	\quad
	\subfigure[MFQNMO]{
		\includegraphics[scale=0.32]{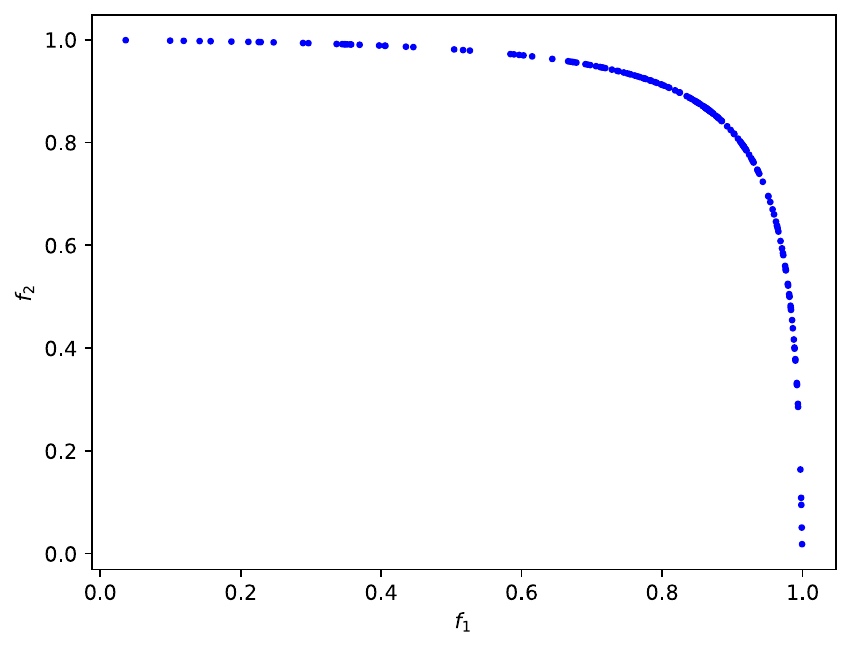}\label{3}
	}
	\quad
	\caption{The approximated nondominated frontiers generated by QNMO, MQNMO, and MFQNMO for the FF1 problem.}  \label{fig4}
\end{figure}		

\begin{figure}[H]
	\centering
	\subfigure[QNMO]{
		\includegraphics[scale=0.32]{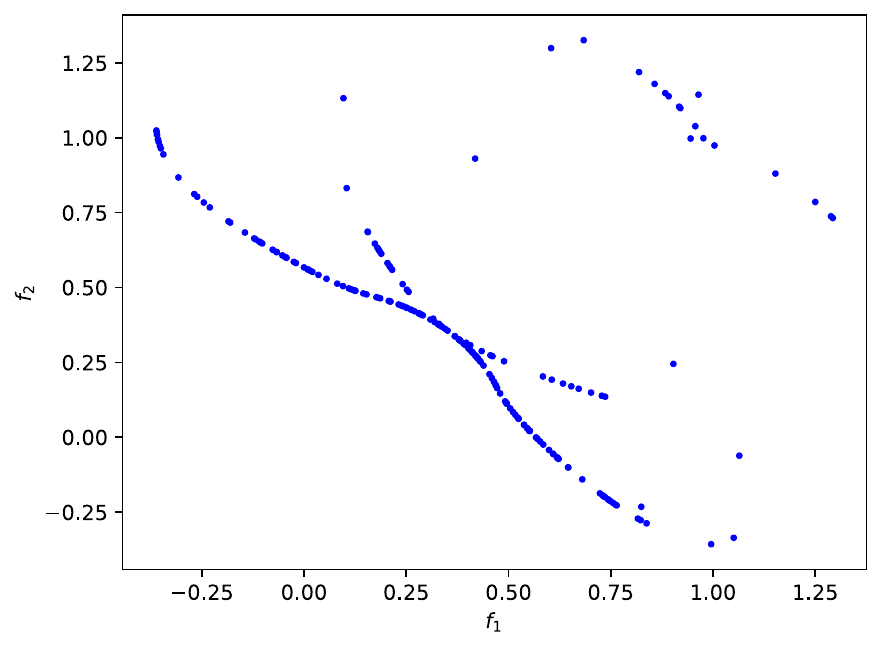} \label{1}
	}
	\quad
	\subfigure[MQNMO]{
		\includegraphics[scale=0.32]{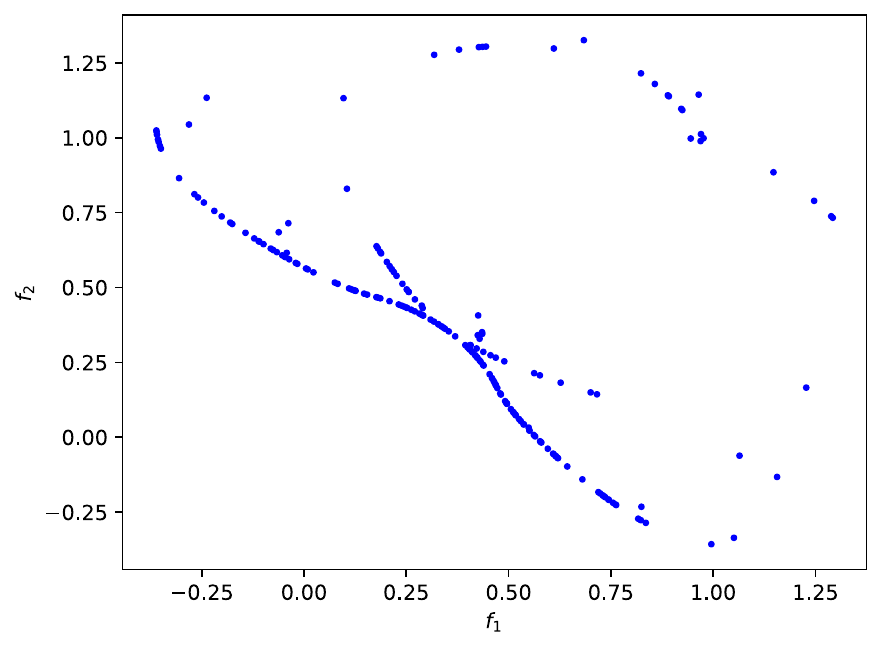}\label{2}
	}
	\quad
	\subfigure[MFQNMO]{
		\includegraphics[scale=0.32]{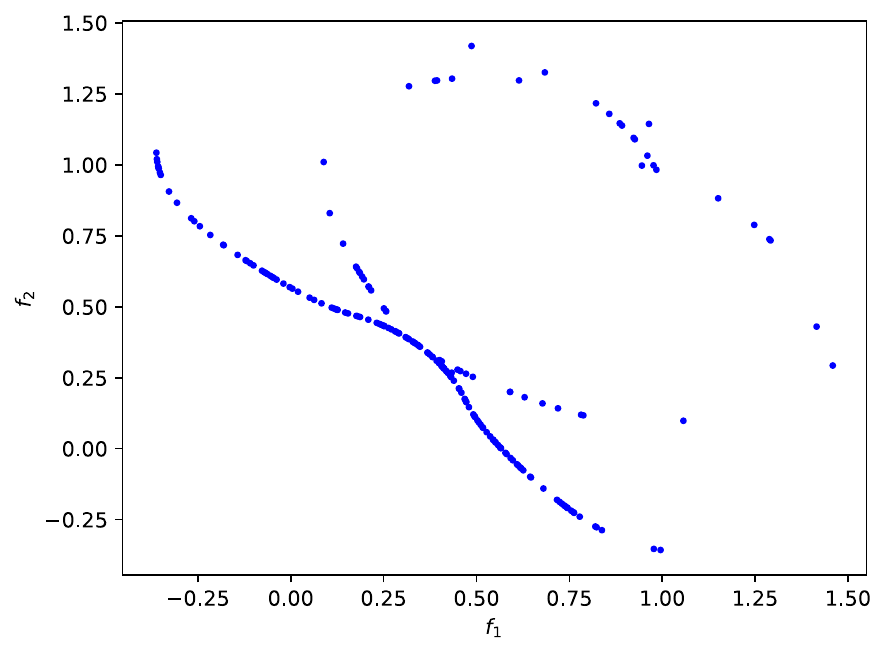}\label{3}
	}
	\quad
	\caption{The approximated nondominated frontiers generated by QNMO, MQNMO, and MFQNMO for the Hil problem.}  \label{fig5}
\end{figure}		

\begin{figure}[H]
	\centering
	\subfigure[QNMO]{
		\includegraphics[scale=0.32]{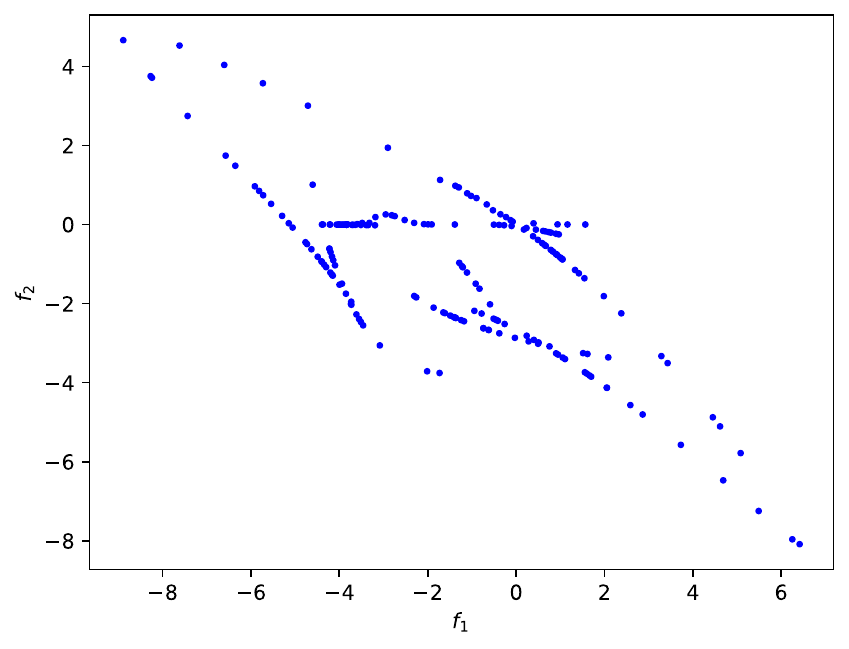} \label{1}
	}
	\quad
	\subfigure[MQNMO]{
		\includegraphics[scale=0.32]{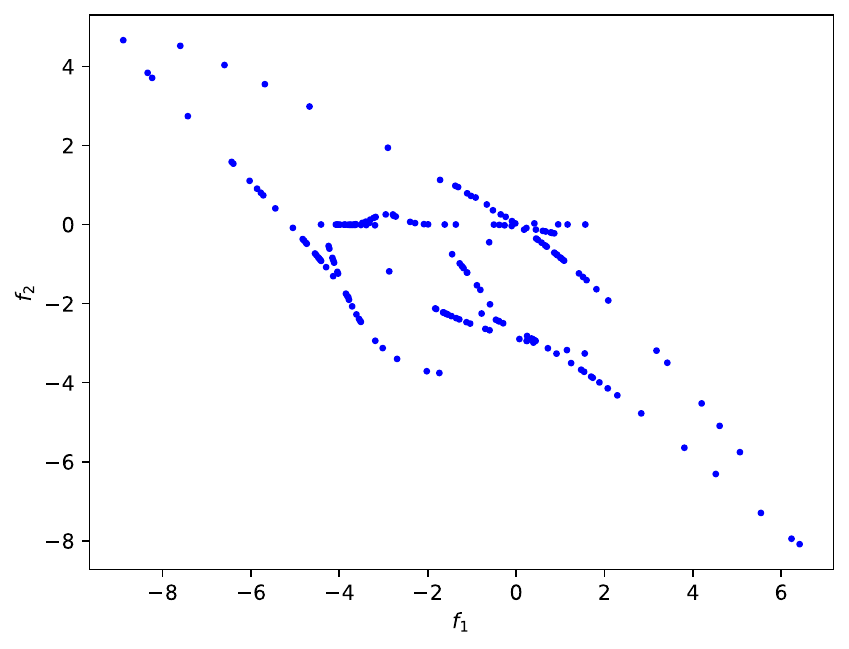}\label{2}
	}
	\quad
	\subfigure[MFQNMO]{
		\includegraphics[scale=0.32]{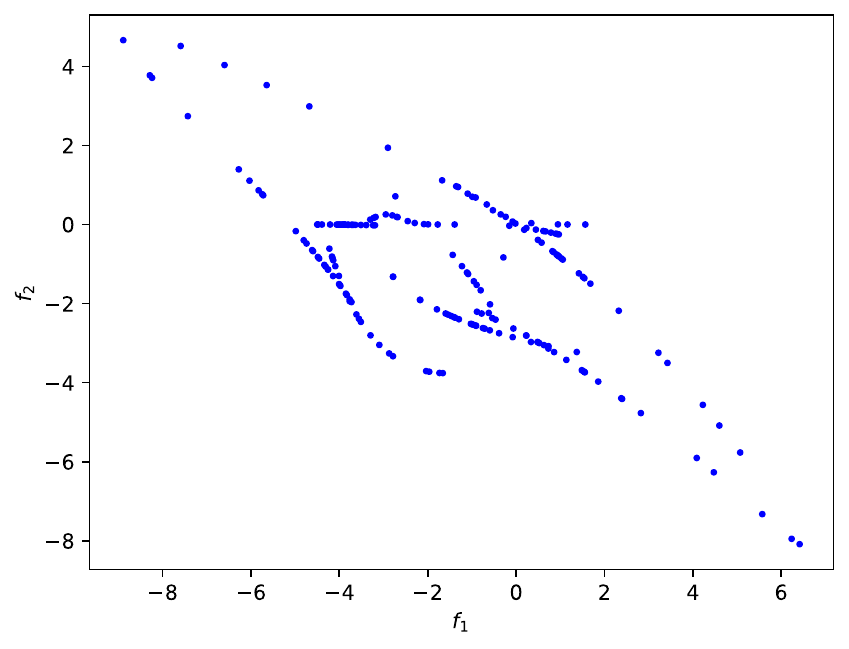}\label{3}
	}
	\quad
	\caption{The approximated nondominated frontiers generated by QNMO, MQNMO, and MFQNMO for the KW2 problem.}  \label{fig6}
\end{figure}

\begin{figure}[H]
	\centering
	\subfigure[QNMO]{
		\includegraphics[scale=0.32]{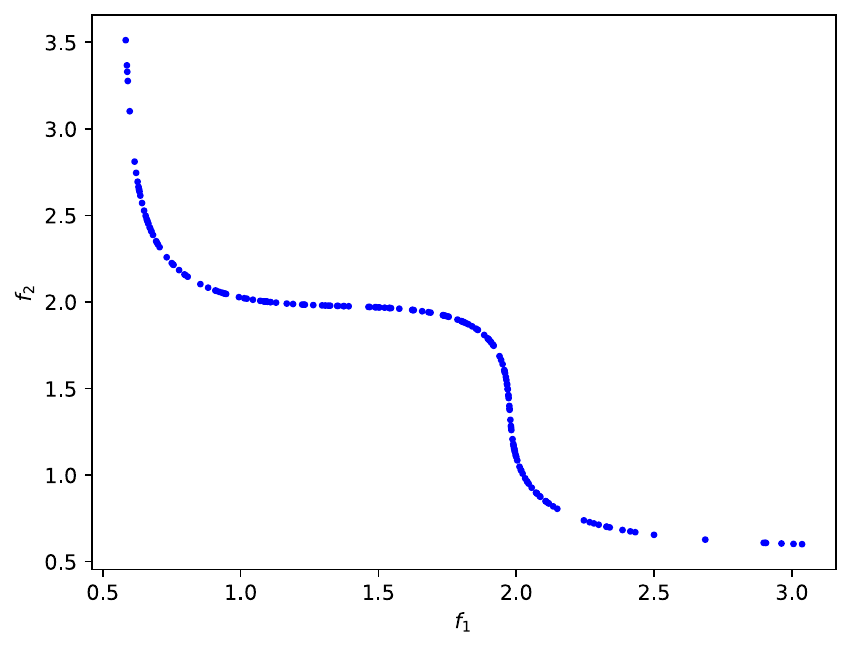} \label{1}
	}
	\quad
	\subfigure[MQNMO]{
		\includegraphics[scale=0.32]{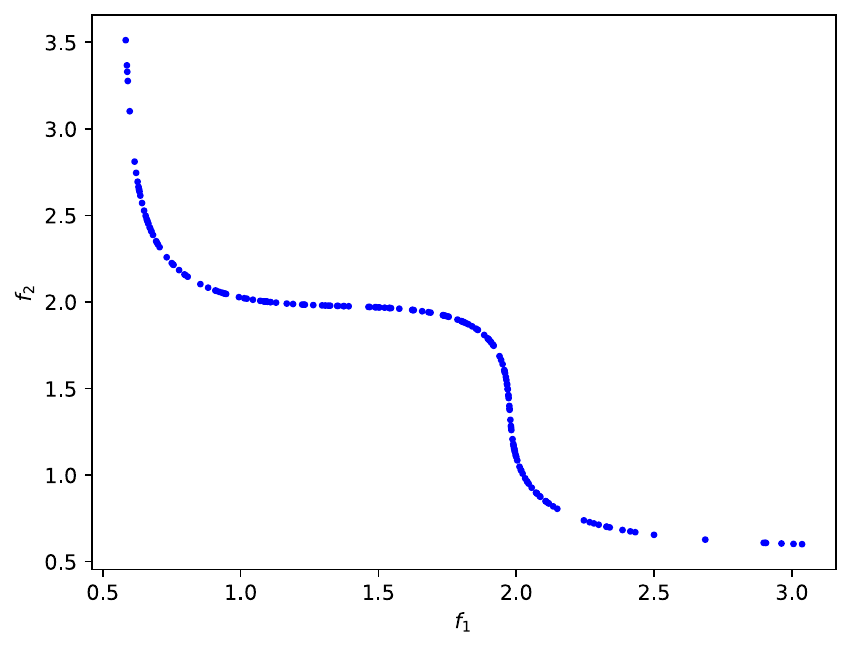}\label{2}
	}
	\quad
	\subfigure[MFQNMO]{
		\includegraphics[scale=0.32]{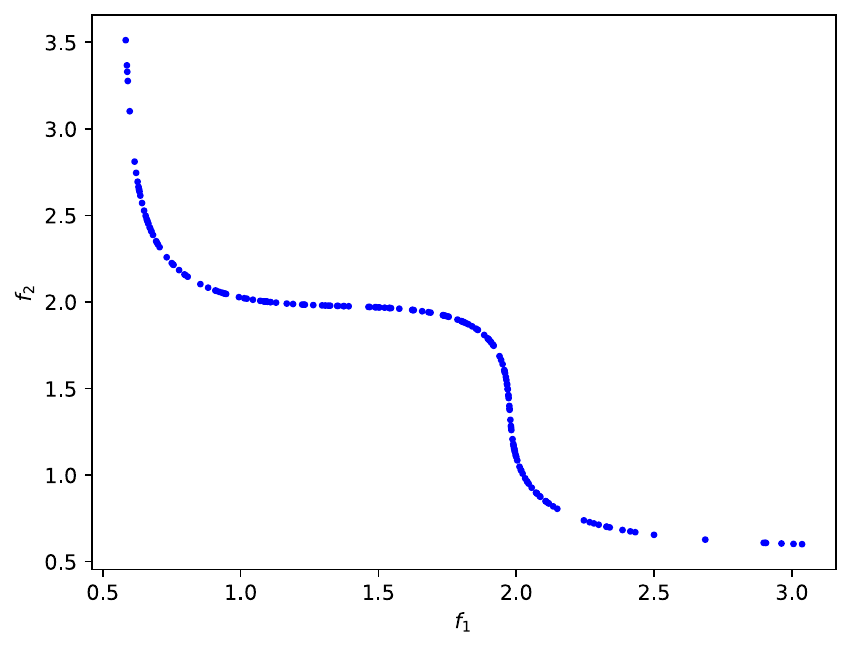}\label{3}
	}
	\quad
	\caption{The approximated nondominated frontiers generated by QNMO, MQNMO, and MFQNMO for the SLCDT1 problem.}  \label{fig7}
\end{figure}	

\begin{figure}[H]
	\centering
	\subfigure[QNMO]{
		\includegraphics[scale=0.32]{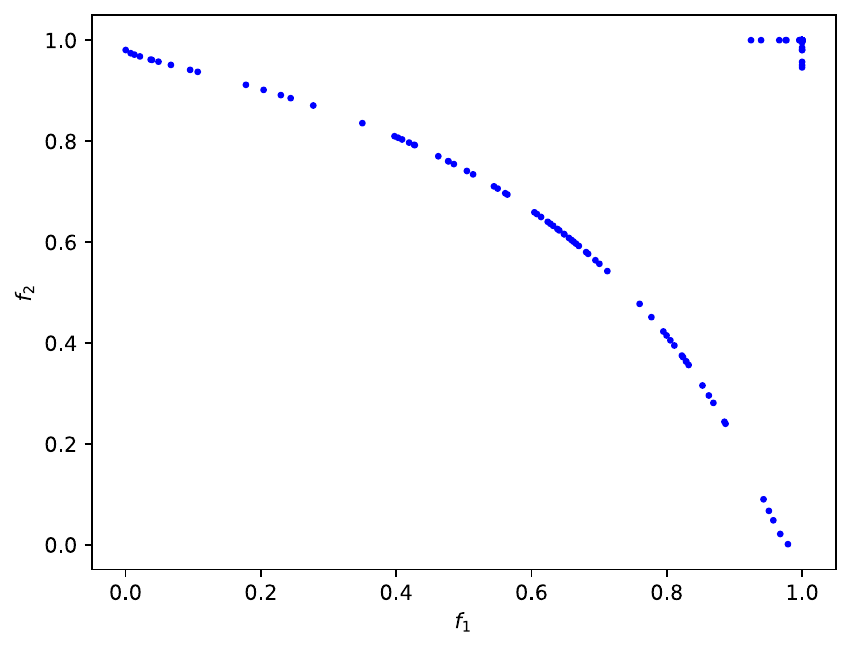} \label{1}
	}
	\quad
	\subfigure[MQNMO]{
		\includegraphics[scale=0.32]{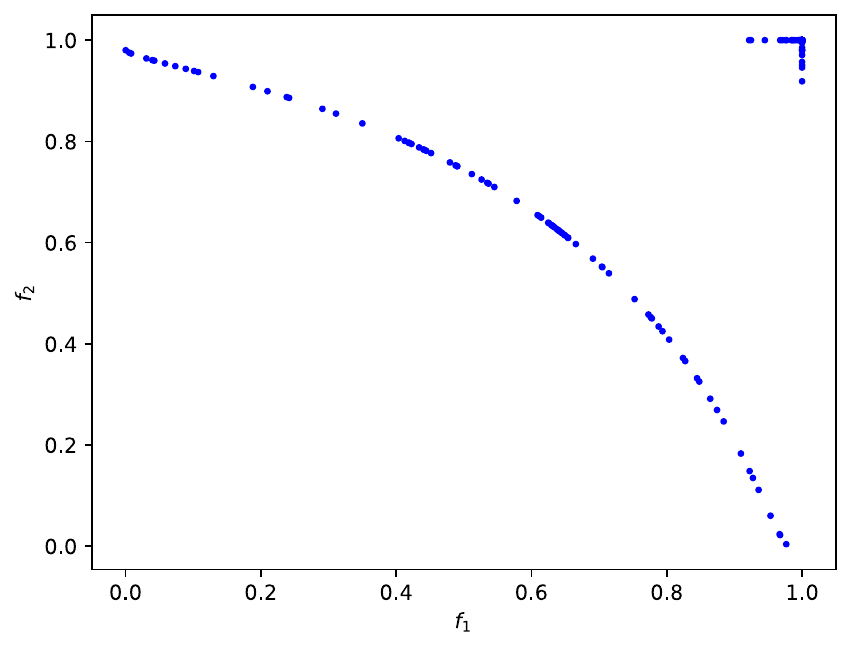}\label{2}
	}
	\quad
	\subfigure[MFQNMO]{
		\includegraphics[scale=0.32]{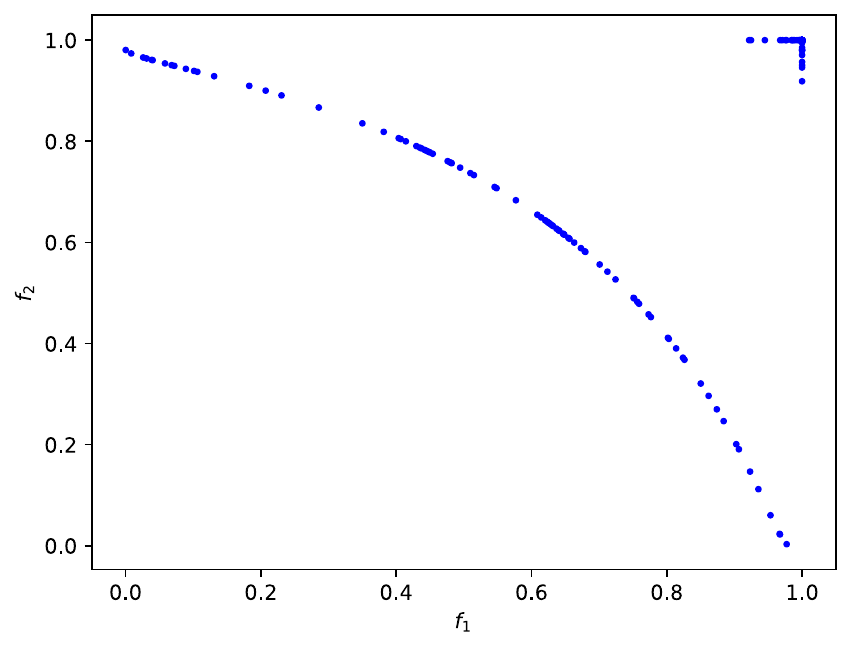}\label{3}
	}
	\quad
	\caption{The approximated nondominated frontiers generated by QNMO, MQNMO, and MFQNMO for the MOP2 problem.}  \label{fig8}
\end{figure}

\begin{figure}[H]
	\centering
	\subfigure[QNMO]{
		\includegraphics[scale=0.32]{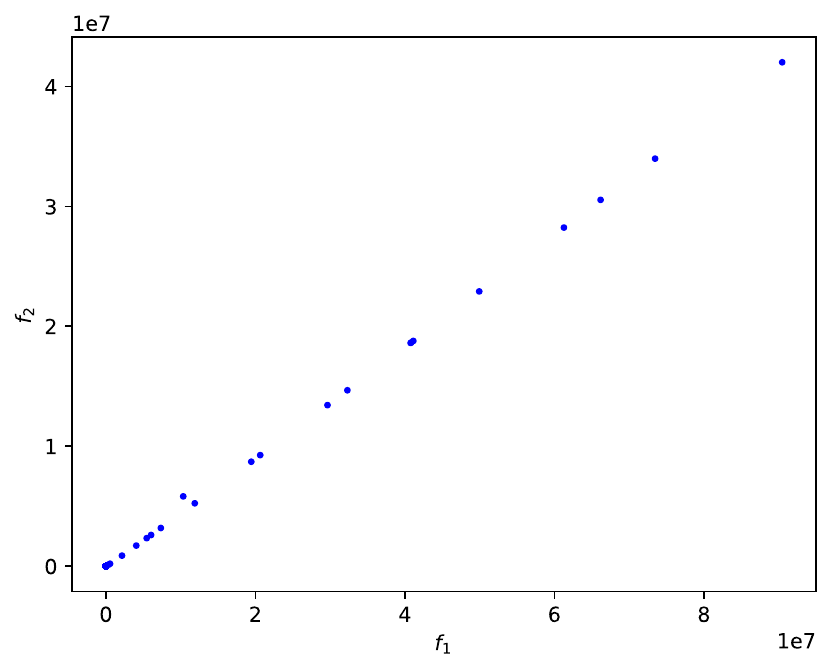} \label{1}
	}
	\quad
	\subfigure[MQNMO]{
		\includegraphics[scale=0.32]{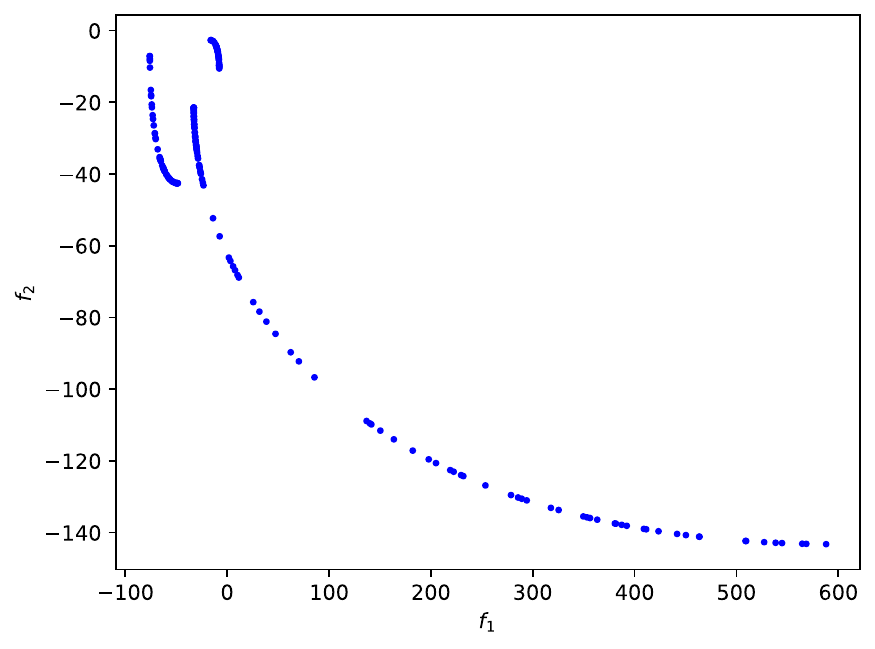}\label{2}
	}
	\quad
	\subfigure[MFQNMO]{
		\includegraphics[scale=0.32]{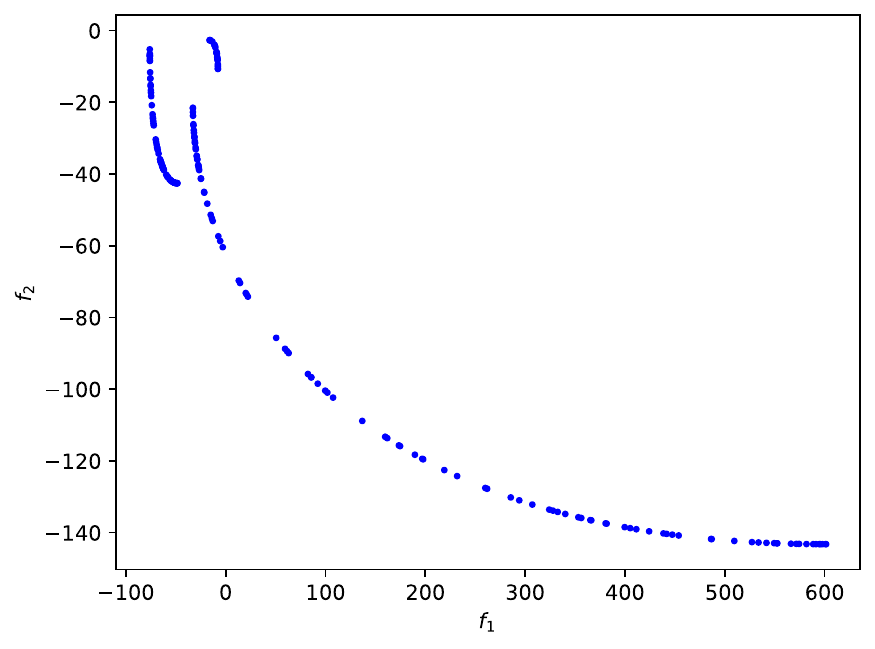}\label{3}
	}
	\quad
	\caption{The approximated nondominated frontiers generated by QNMO, MQNMO, and MFQNMO for the SK1 problem.}  \label{fig9}
\end{figure}	

\begin{figure}[H]
	\centering
	\subfigure[QNMO]{
		\includegraphics[scale=0.4]{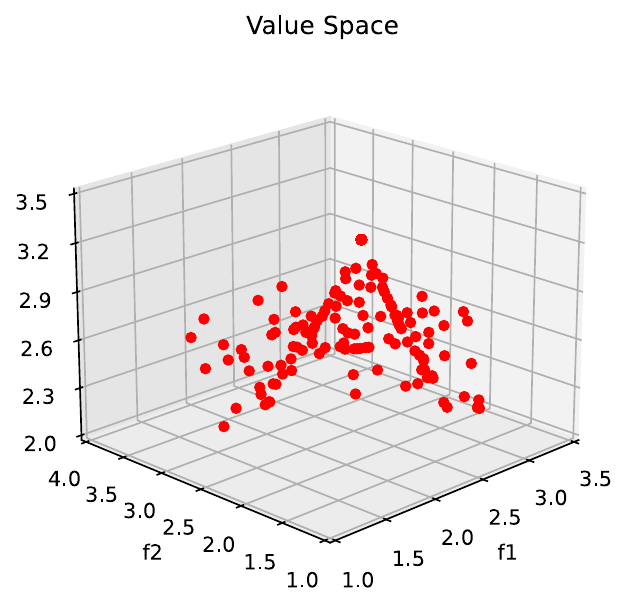} \label{1}
	}
	\quad
	\subfigure[MQNMO]{
		\includegraphics[scale=0.4]{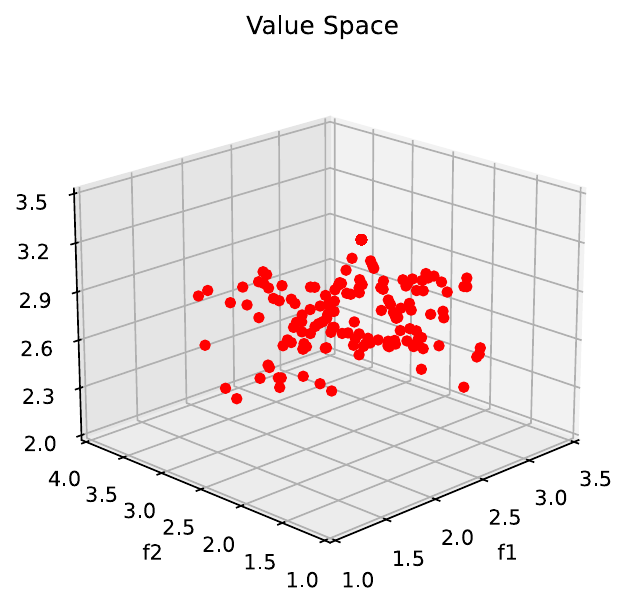}\label{2}
	}
	\quad
	\subfigure[MFQNMO]{
		\includegraphics[scale=0.4]{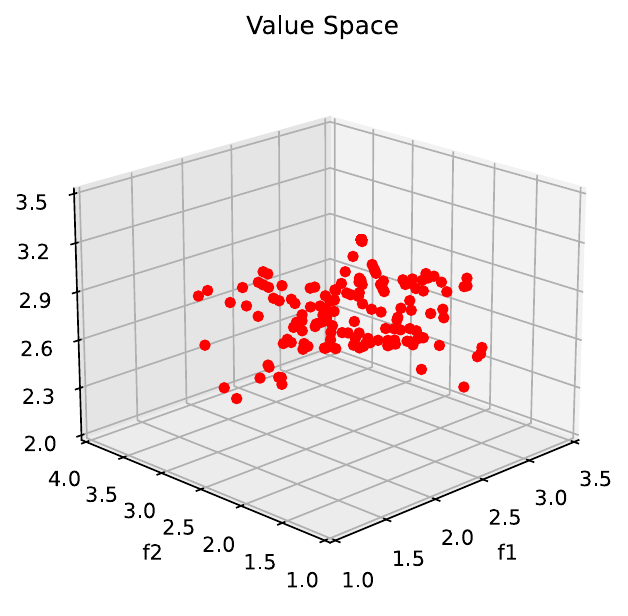}\label{3}
	}
	\quad
	\caption{The approximated nondominated frontiers generated by QNMO, MQNMO, and MFQNMO for the LDTZ problem.}  \label{fig10}
\end{figure}	
\begin{figure}[H]
	\centering
	\subfigure[QNMO]{
		\includegraphics[scale=0.32]{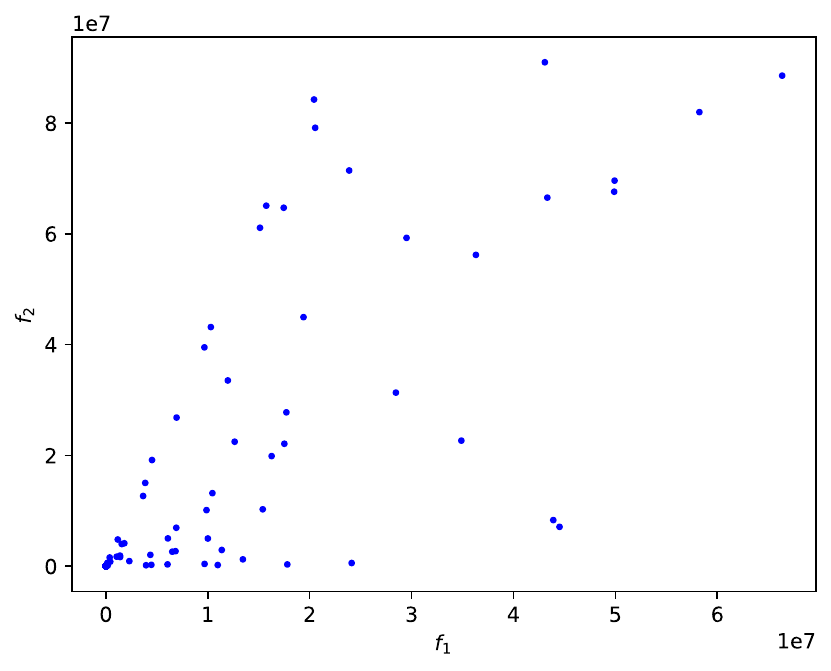} \label{1}
	}
	\quad
	\subfigure[MQNMO]{
		\includegraphics[scale=0.32]{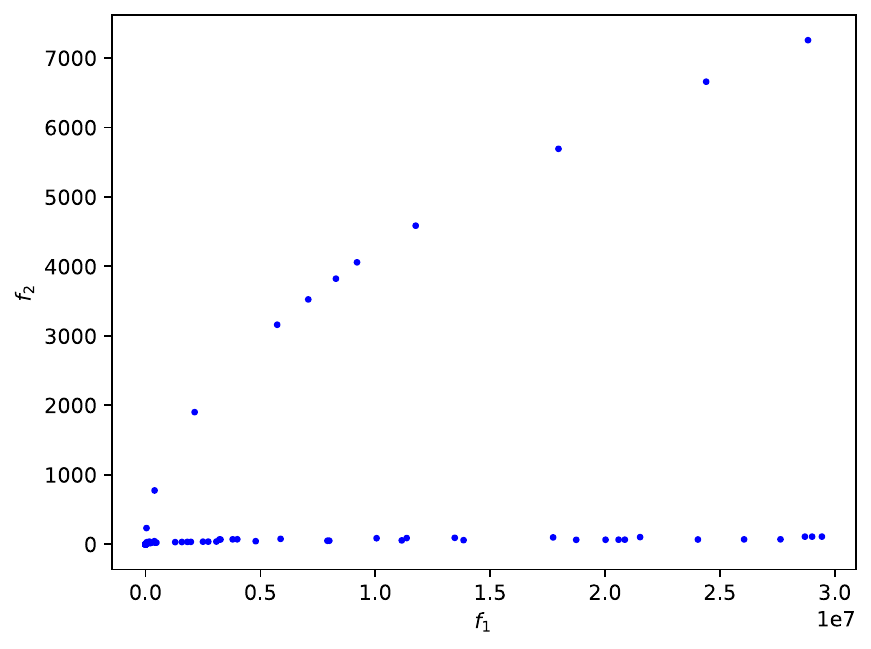}\label{2}
	}
	\quad
	\subfigure[MFQNMO]{
		\includegraphics[scale=0.32]{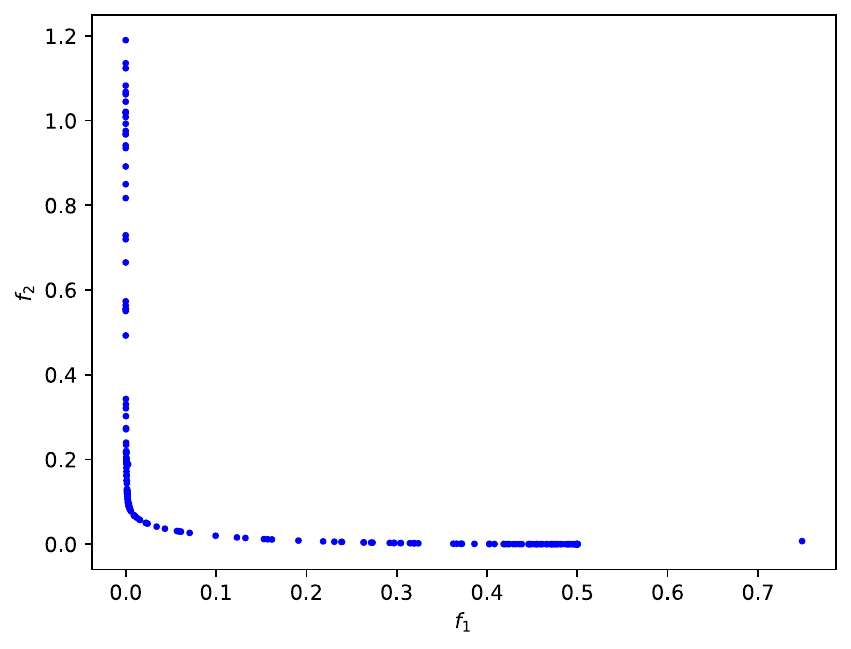}\label{3}
	}
	\quad
	\caption{The approximated nondominated frontiers generated by QNMO, MQNMO, and MFQNMO for the AP3 problem.}  \label{fig11}
\end{figure}

	\section{Conclusion} \label{sec5:Conclusion}
	
	\indent We propose a modified BFGS-type method based on function information for nonconvex multiobjective optimization problems (MFQNMO), in which the BFGS-type formula is updated by using the function information. To obtain the global convergence property, we only need the smoothness of the objective function. Compared with \cite{Global quasi-Newton}, our work has the following advantages.\\
	\indent (1)\ Instead of approximating the Hessian matrix of all objective functions, we use a common matrix to capture the curvature information of the objective function, which balances the effectiveness of the algorithm and the computational cost.\\
	\indent (2)\ For nonconvex functions, most versions of quasi-Newton methods may not converge, so the MFQNMO method proposed in this paper for nonconvex unconstrained problems is meaningful.\\
	\indent (3)\ Compared with \cite{Global quasi-Newton}, the update of the BFGS-type formula does not depend on the selection of parameters, but only needs to use the information of past iterations.\\
	\indent In our future work, we will consider applying the method to nonconvex constrained optimization problems.

\end{document}